\documentclass[10pt]{amsart}
\usepackage{appendix}
\usepackage{graphicx,enumitem,dsfont}
\usepackage[usenames]{color}
\usepackage[colorlinks]{hyperref}
\usepackage{subfigure}

\newcommand{\R}{\mathbb{R}}
\newcommand{\dd}{\mathbf{d}}

\newcommand{\U}{{\mathbf{u}}}
\newcommand{\B}{{\mathbf{B}}}

\newcommand{\X}{{\mathbf{x}}}
\newcommand{\E}{{\mathbf{E}}}
\newcommand{\J}{{\mathbf{J}}}
\newcommand{\s}{{\mathbf{S}}}
\newcommand{\n}{{\mathbf{n}}}

\newcommand{\Div}{\mathrm{div}}
\newcommand{\Curl}{\mathrm{curl}}
\newcommand{\tildeCurl}{\widetilde{\Curl}}

\newcommand{\diag}[1]{\mathrm{diag}\left(#1\right)}
\newcommand{\bD}{\mathfrak{d}}
\newcommand{\bP}{\mathbf{P}}
\newcommand{\mL}{\mathcal{L}}
\newcommand{\mR}{\mathcal{R}}
\newcommand{\mD}{\mathcal{D}}
\newcommand{\mU}{\mathcal{U}}
\newcommand{\mA}{\mathcal{A}}
\newcommand{\mB}{\mathcal{B}}
\newcommand{\crl}{\frak{curl}}
\newcommand{\bV}{\mathbf{V}}

\newcommand{\Comment}[1]{}

\newcommand{\eps} {\varepsilon}
\renewcommand{\i}{\ifmmode\mathit{\mathchar"7010 }\else\char"10 \fi}
\renewcommand{\j}{\ifmmode\mathit{\mathchar"7011 }\else\char"11 \fi}

\newcommand{\seq}[1]{\left\{#1\right\}}

\newcommand{\Dx}{\Delta x}
\newcommand{\Dy}{\Delta y}
\newcommand{\Dz}{\Delta z}

\newcommand{\norm}[1]{\left\|#1\right\|}
\newcommand{\abs}[1]{\left|#1\right|}

\newtheorem{theorem}{Theorem}[section]

\newtheorem{lemma}{Lemma}[section]

\newtheorem{remark}{Remark}[section]
\theoremstyle{definition} 

\newtheorem*{maintheorem*}{Main Theorem}

\allowdisplaybreaks

\numberwithin{equation}{section}
\numberwithin{figure}{section}
\numberwithin{table}{section}

\newcounter{asnr}

{\ifnum\value{asnr}=0 \stepcounter{asnr} 
  \begin{enumerate}[label=\textbf{A}.\arabic{enumi}]
    \else
    \begin{enumerate}[label=\textbf{A}.\arabic{enumi},resume] \fi}
{\end{enumerate}}

\title[SBP-SAT schemes for Magnetic induction equations]{Higher order
  finite difference schemes for the Magnetic Induction
  equations with resistivity}

\author[U. Koley]{U. Koley} \address[Ujjwal Koley]{\newline Centre of
  Mathematics for Applications (CMA) \newline University of
  Oslo\newline P.O. Box 1053, Blindern\newline N--0316 Oslo, Norway}
\email[]{ujjwalk@cma.uio.no}

\author[S. Mishra]{S. Mishra} \address[Siddhartha Mishra]{\newline
  Centre of Mathematics for Applications (CMA) \newline University of
  Oslo\newline P.O. Box 1053, Blindern\newline N--0316 Oslo, Norway}
\email[]{siddharm@cma.uio.no}

\author[N. H. Risebro]{N. H. Risebro} \address[Nils Henrik
Risebro]{\newline Centre of Mathematics for Applications (CMA)
  \newline University of Oslo\newline P.O. Box 1053, Blindern\newline
  N--0316 Oslo, Norway} \email[]{nilshr@math.uio.no}

\author[M. sv\"{a}rd]{M. Sv\"{a}rd} \address[Magnus Sv\"{a}rd]
{\newline Centre of Mathematics for Applications (CMA)
  \newline University of Oslo\newline P.O. Box 1053, Blindern\newline
  N--0316 Oslo, Norway} \email[]{magnus.svard@cma.uio.no}

\date{\today}

\begin{document}

\begin{abstract}
  In this paper, we design high order accurate and stable finite
  difference schemes for the initial-boundary value problem, associated
  with the magnetic induction equation with resistivity. We use Summation-By-Parts
  (SBP) finite difference operators to approximate spatial derivatives
  and a Simultaneous Approximation Term (SAT) technique for
  implementing boundary conditions. The resulting schemes are shown to be energy stable. Various numerical experiments
  demonstrating both the stability and the high order of accuracy of the
  schemes are presented.
\end{abstract}

\maketitle

\section{Introduction}
\label{sec:intro}
Many interesting problems in astrophysics and engineering involve evolution of macroscopic plasmas, modeled by the equations of MagnetoHydroDynamics (MHD). These equations (\cite{BIS1}) are a system of convection-diffusion equations with the magnetic resistivity and heat conduction playing the role of diffusion. Many applications like plasma thrusters for deep space propulsion and electromagnetic pulse devices (\cite{tg}) involve small (but non-zero) values of the magnetic resistivity. Hence, the design of efficient numerical methods for the resistive MHD equations is essential for simulating some of the afore mentioned models.

Numerical study of the ideal MHD equations (where magnetic resistivity and other diffusions terms are neglected) has witnessed considerable progress in recent years and a variety of numerical methods are available (see \cite{frsid1} for a review of the available literature). The design of numerical schemes for the resistive MHD equations has not reached the same stage of maturity as the presence of magnetic resistivity complicates the design of stable methods even further. Given the formidable difficulties, study of prototypical sub-models (that mirror some, but not all of the difficulties of the resistive MHD equations) can be a useful guide for obtaining robust methods for the resistive MHD equations.

In this paper, we consider the magnetic induction equations with resistivity. Recent papers (\cite{unsm,fkrsid1}) have pointed out the role that the magnetic induction equation (without resistivity) plays in the design of numerical schemes for the ideal (inviscid) MHD equations. The induction equations with resistivity can play a similar role for designing stable methods for the resistive MHD equations. Our goal in this paper is to design stable and high-order
accurate numerical schemes for the magnetic induction equations with
resistivity.

We start with a brief description of how the equations are derived.
In a moving medium, the time rate of change of the magnetic flux
across a given surface $\s$ bounded by curve $\partial\s$ is
given by (see \cite{Pano}):
\begin{equation*}
  \frac {d}{dt} \int\limits_{S} \B\cdot \dd\s = \int\limits_{S} \frac
  {\partial \B}{\partial t}\cdot \dd\s +\oint\limits_{\partial S} \B
  \times \U \cdot dl + \int\limits_{S} ({\rm div} (\B))\U \cdot \dd\s
  + \epsilon \oint\limits_{\partial S} \J \cdot dl,
\end{equation*}
where the unknown $\B=\B(\X,t)\in\R^3$ denotes the magnetic field,
$\J=\J(\X,t)\in\R^3$ the current density and $\X=(x,y,z)$ are the spatial coordinates. The current density is given by: $\J= \Curl(\B)$. The parameter $\epsilon$ denotes the magnetic resistivity, and $\U(\X,t)$ the
(given) velocity field.

Using Faraday's law:
\begin{equation}\label{eq:Faraday}
  -\frac {d}{dt}\int\limits_{S} \B \cdot \dd\s =
  \oint\limits_{\partial S}\E'\cdot dl,
\end{equation}
 Stokes' theorem, the fact that the electric field $\E'=0$ in a co-moving frame and  $ \E' = \E + \U \times \B $ we obtain, 
\begin{equation}
  \label{eq:viscousinduction}
  \frac {\partial \B}{\partial t} + \Curl(\B \times \U) = - \U {\rm
    div}(\B)- \epsilon\Curl(\Curl(\B)). 
\end{equation}
Magnetic monopoles have never been observed in nature. As a
consequence, the magnetic field is always assumed to be divergence
free, i.e., $\Div(\B)=0$. Using this constraing in
\eqref{eq:viscousinduction}, we obtain the system:
\begin{equation}
  \label{eq:Maxwell3D}
  \begin{aligned}
    \partial_t \B + \Curl(\B \times \U) &= - \epsilon\Curl(\Curl(\B)),\\
    \Div(\B) &= 0.
  \end{aligned}
\end{equation}
The above equation is an example of a convection-diffusion equation. The version obtained by taking zero resistivity ($\epsilon=0$) in \eqref{eq:Maxwell3D} is termed the magnetic induction equation (\cite{TF1}). A standard way to obtain a bound on the solutions of convection-diffusion equations like \eqref{eq:Maxwell3D} is to use the energy
method.  However \eqref{eq:Maxwell3D} is not symmetrizable. Consequently it may not be possible to obtain an energy estimate for this
system. 

On the other hand, \eqref{eq:viscousinduction} is symmetrizable. We use the following vector identity 
\begin{align*}
  \Curl(\B\times \U)&=\B\Div \U - \U\Div (\B) +
  \left(\U\cdot\nabla\right) \B -
  \left(\B\cdot\nabla\right) \U\\
  &=\left(u^1\B\right)_x+\left(u^2\B\right)_y+\left(u^3\B\right)_z -
  \U\Div(\B) - (\B\cdot\nabla)\U,
\end{align*}
and rewrite \eqref{eq:viscousinduction} in the form,
\begin{equation}
  \begin{aligned}
    \partial_t\B + \left(\U\cdot\nabla\right)\B&=
    -\B(\Div\U)+(\B\cdot\nabla)\U -  \epsilon\Curl(\Curl(\B))\\
    &=M(D\U)\B  -  \epsilon\Curl(\Curl(\B)),
  \end{aligned}
  \label{eq:viscousinduc1}  
\end{equation}
where the matrix $M(D\U)$ is given by
$$
M(D\U)=
 \begin{pmatrix}
   -\partial_y u^2 - \partial_z u^3 & \partial_y u^1 & \partial_z
   u^1\\
   \partial_x u^2 & -\partial_x u^1-\partial_z u^3 & \partial_z u^2\\
   \partial_x u^3 & \partial_y u^3 & -\partial_x u^1 - \partial_y u^2
  \end{pmatrix}.
  $$
  Introducing the matrix,
  $$
  C=-\begin{pmatrix}
    \partial_x u^1 & \partial_y u^1 & \partial_z u^1 \\
    \partial_x u^2 & \partial_y u^2 & \partial_z u^2 \\
    \partial_x u^3 & \partial_y u^3 & \partial_z u^1
\end{pmatrix},
$$
\eqref{eq:viscousinduction} can also be written in the following
form,
\begin{equation}
  \label{eq:viscousinduc2}
  \partial_t\B + \partial_x\left(A^1\B\right)
  +\partial_y\left(A^2\B\right)+\partial_z\left(A^3\B\right)+C\B
  =- \epsilon\Curl(\Curl(\B)), 
\end{equation}
where $A^i=u^iI$ for $i=1,2,3$. Note that the symmetrized matrices in
\eqref{eq:viscousinduc2} are diagonal and that the coupling in the
equations are through both the lower order source terms and the viscous terms.

Furthermore,  taking the divergence of both sides of
\eqref{eq:viscousinduction} we obtain,
\begin{equation}
  \label{eq:divt}
  (\Div(\B))_t + \Div\left(\U\Div(\B)\right) = 0.
\end{equation}
Hence, if $\Div(\B(\X,0))=0$, it follows that $\Div(\B(\X,t))=0$ for
$t>0$. This implies that all the above forms \eqref{eq:viscousinduc2},
\eqref{eq:Maxwell3D} and \eqref{eq:viscousinduction} are equivalent
(at least for smooth solutions).

Although the magnetic induction equations with resistivity are linear, the
coefficients are functions of $\X$ and $t$. Therefore, closed form
solutions are not available and we must resort to numerical methods
in order to calculate (approximate) solutions. Consequently, it is
important to design efficient numerical methods for these equations. 

As mentioned before the magnetic induction equations is a sub-model in
the resistive MHD equations. Hence, design of stable and high-order accurate
numerical schemes for the viscous induction equations can lead to
robust schemes for the non-linear resistive MHD equations.

The presence of the $\Div(\B)= 0$ constraint leads to numerical
difficulties. Small divergence errors may
change the nature of results from numerical
simulations (see \cite{BlBr1,Toth1} for details on the role of divergence in ideal MHD codes). Our approach to treating the constraint follows the method developed in \cite{Pano,fkrsid1,unsm} and involves discretizing 
\eqref{eq:viscousinduc2}. TFurthermore, a proper discretization of the symmetric form \eqref{eq:viscousinduc2} yields energy estimates. These estimates are vital in proving
existence of weak solutions.  We will approximate
spatial derivatives by second and fourth order SBP (``Summation By
Parts'') operators. The boundary conditions of both the Dirichlet and mixed type are weakly imposed by
using a SAT (``Simultaneous Approximation Term''). This work is an extension of the SBP-SAT schemes for the case without resistivity ($\epsilon=0$) found in a recent paper \cite{unsm}.

We would like to emphasize that other numerical frameworks like mixed finite elements, discrete duality finite volume or mimetic finite differences might also lead to stable schemes for approximating the induction equations with resistivity. However, we are not aware of any papers that have approximated the resistive induction equations with these approaches.

The rest of this paper is organized as follows: In
Section~\ref{sec:cp}, we state the energy estimate for the
initial-boundary value problem corresponding to
\eqref{eq:viscousinduc1} in order to motivate the proof of stability
for the scheme. This is done for  mixed type and Dirichlet boundary
conditions. In
Section~\ref{sec:scheme}, we present the SBP-SAT scheme and show its
stability with both Dirichlet and mixed boundary conditions.
Numerical experiments are presented in Section~\ref{sec:numex} and
conclusions from this paper are drawn in Section~\ref{sec:conc}.

\section{The Continuous problem} \label{sec:cp} 
For simplicity and notational convenience, we restrict
ourselves to two spatial dimensions in the remainder of this paper. Extending our results to three
dimensions is straightforward.

In two dimensions, \eqref{eq:viscousinduc1} reads
\begin{equation}
  \label{eq:main}
  \begin{gathered}
    \B_t + \Lambda_1\B_x + \Lambda_2\B_y - C\B = -\epsilon
    \nabla\times\left(\nabla\times \B\right),\\
    \text{where}\quad -\nabla\times\left(\nabla\times \B\right) =
    \left[
      \begin{aligned}
        - \left((B^2)_{xy}-(B^1)_{yy}\right)\\
        \left((B^2)_{xx}-(B^1)_{xy}\right)
      \end{aligned}
    \right],
  \end{gathered}
\end{equation}
and
\begin{equation*}
  \Lambda_1=
  \begin{pmatrix}
    u^1 & 0 \\
    0 & u^1
  \end{pmatrix},\quad
  \Lambda_2=
  \begin{pmatrix}
    u^2 & 0 \\
    0 & u^2
  \end{pmatrix},\quad
  C=\begin{pmatrix}
    -\partial_y u^2 & \partial_y u^1 \\
    \partial_x u^2 & -\partial_x u^1
  \end{pmatrix},
\end{equation*}
with $\B = \left( B^1, B^2\right)^T$ and $\U = \left(u^1,u^2\right)^T$
denoting the magnetic and velocity fields respectively.  Throughout
this paper, we consider \eqref{eq:main} in a smooth domain
$\Omega$.  One can extend our results to general piecewise smooth
boundaries by a standard procedure.  We
augment \eqref{eq:main} with initial conditions,
\begin{equation}
  \label{eq:initial}
  \B(\X,0) = \B_0(\X), \quad \X\in \Omega,
\end{equation}
and Dirichlet or mixed boundary conditions with homogeneous boundary
data. The Dirichlet boundary conditions are given as
\begin{equation}
  \label{eq:dirichletboundary}
  \B(\X,t)=0,\quad\text{$\X\in\partial \Omega$.}
\end{equation}
In order to specify the mixed boundary conditions, we need some
notation.  Let $\n(\X)$ denote the outward pointing unit normal at a
point $\X\in \partial \Omega$. Define 
\begin{align}
  \tildeCurl(\B)\times \n &=\begin{pmatrix} -n^2(B^2_x-B^1_y)\\
    n^1(B^2_x-B^1_y)
  \end{pmatrix}
  \  \text{and}\\
  \tildeCurl(\B) &= B^2_x- B^1_y,
\end{align}
for $ \B= (B^1,B^2)$ and $\n=(n^1,n^2)$.  Furthermore, let
$\partial\Omega_{\mathrm{in}}$ denote the part of $\partial\Omega$
where the characteristics are incoming, i.e.,
$$
\partial\Omega_{\mathrm{in}}= \seq{\X\in\partial\Omega\;\bigm|\;
\n(\X)\cdot \U(\X) <0}.
$$
With this notation, the mixed boundary conditions read
\begin{equation}
  \label{eq:mixedboundary}
  \begin{aligned}
     \beta(\n(\X)\cdot \U(\X))\, \B(\X,t) + \tildeCurl(\B(\X,t))\times
     \n(\X)&=0, \quad \text{for $\X 
       \in \partial\Omega_{\mathrm{in}}$},
     \\
     \tildeCurl(\B(\X,t))\times \n(\X) &= 0, \quad \text{for
      $\X\in \partial\Omega \setminus \partial\Omega_{\mathrm{in}}$,}
  \end{aligned}
\end{equation}
where $\beta<-1/(2\epsilon)$ is a given number. 

In order to motivate the complicated calculations required to show
stability in the discrete case, we start by explaining how stability
is proved in the continuous case. We assume that the
solution, $\B$, is sufficiently regular for our calculations to make
sense.
\begin{theorem}
  \label{thm:contidirichlet}
  Let $\B(\X,t)$ be a solution of the  problem
  \eqref{eq:main} and \eqref{eq:initial} with boundary conditions 
  \eqref{eq:dirichletboundary} or \eqref{eq:mixedboundary}. There exist positive constants
  $\alpha$ (depending on $\U$ and its first derivatives) and $K$, such
  that
  $$
  \norm{\B(\cdot,t)}_{L^2(\Omega)}^2 + \epsilon \norm{\tildeCurl
    (\B(\cdot,t))}_{L^2(\Omega)}^2 \le K e^{\alpha t}
  \norm{\B_0}_{L^2(\Omega)}^2.
  $$
\end{theorem}
\begin{proof}
  Multiplying \eqref{eq:main} by $\B^T$ and then integrating in space,
  we get
  \begin{multline*}
    \int_\Omega \B^T \partial_t\B \,dxdy  + \int_\Omega \left (\B^T
      \Lambda_1\B_x  + \B^T \Lambda_2\B_y - \B^TC\B \right)\,dxdy  \\
     = -\epsilon \int_\Omega \B^T \left
      (\nabla\times(\nabla\times\B)\right)\,dxdy,
  \end{multline*}
  which implies,
  \begin{align*}
    \frac{1}{2}& \frac{\partial}{\partial t}\int_\Omega \B^2 \,dxdy + \epsilon
    \int_\Omega \left
      (\tildeCurl(\B)\right)^2 \,dxdy\\
    & = \frac{1}{2}\int_\Omega \B^T\left(2C+\Div(\U)\right)\B\,dxdy -
    \frac{1}{2}\int_{\partial\Omega} \B^2 \left ( \U \cdot \n
    \right)\,ds + \epsilon \int_{\partial\Omega} \left( \B \cdot
      (\tildeCurl(\B) \times \n) \right)\,ds,
  \end{align*}
  so that,
  \begin{align*}
    \frac{1}{2}& \frac{\partial}{\partial t}\int_\Omega \B^2 \,dxdy + \epsilon
    \int_\Omega \left
      (\tildeCurl(\B)\right)^2 \,dxdy\\
    & \le \alpha \int_\Omega \B^2 \,dxdy -
    \frac{1}{2}\int_{\partial\Omega} \B^2 \left ( \U \cdot \n
    \right)\,ds + \epsilon \int_{\partial\Omega} \left( \B \cdot
      (\tildeCurl(\B) \times \n) \right)\,ds.
  \end{align*}
  From the above relation, we see that applying Dirichlet boundary
  conditions, \eqref{eq:dirichletboundary}, and integrating in time gives the
  required result.
  For the mixed boundary conditions \eqref{eq:mixedboundary}, we split the
  boundary into $\partial\Omega_{\mathrm{in}}$ and
  $\partial\Omega\setminus\partial\Omega_{\mathrm{in}}$. This yields
  \begin{align*}
    \frac{1}{2}& \frac{\partial}{\partial t}\int_\Omega \B^2 \,dxdy + \epsilon
    \int_\Omega \left
      (\tildeCurl(\B)\right)^2 \,dxdy\\
    &\le \alpha \int_\Omega \B^2 \,dxdy -
    \frac{1}{2}\left(\int_{\partial\Omega_{\mathrm{in}}}
      +\int_{\partial\Omega\setminus\partial\Omega\mathrm{in}}\right) 
    \B^2 \left ( \U \cdot \n \right)\,ds \\
    & \qquad \qquad + \epsilon \left(\int_{\partial\Omega_{\mathrm{in}}}
      +\int_{\partial\Omega\setminus\partial\Omega\mathrm{in}}\right)
    \left( \B \cdot (\tildeCurl(\B) \times \n) \right)\,ds.
  \end{align*}
  Rearranging the above relation and applying mixed boundary
  conditions \eqref{eq:mixedboundary}, remembering that $\beta<-1/(2\epsilon) $, we get
  \begin{align*}
    \frac{1}{2}& \frac{\partial}{\partial t}\int_\Omega \B^2 \,dxdy  + \epsilon \int_\Omega \left
      (\tildeCurl(\B)\right)^2 \,dxdy \\
    & \le  \alpha \int_\Omega \B^2 \,dxdy
    -  \int_{\partial\Omega_{\mathrm{in}}}
    \left(\frac{1}{2}+\epsilon\beta\right)
    \B^2 \left( \U \cdot \n \right)\,ds.
  \end{align*}
  From the above relation, after integrating in time then we have the
  required result.
\end{proof}

\section{Semi-discrete Schemes}
\label{sec:scheme}
To simplify the treatment of the boundary terms we let the computational domain $\Omega$ be the unit square. A justification for this will be provided at the end of this section.

The SBP finite difference schemes for one-dimensional derivative approximations are as follows. Let $[0,1]$ be the domain discretized with $x_j=j\Dx$, $j=0,\dots,N-1$. A scalar grid function is defined as $w=(w_0,...w_{N-1})$. To approximate
$\partial_x w$ we use a summation-by-parts operator $D_x = P^{-1}_x
Q_x$, where $P_x$ is a diagonal positive $N\times N$ matrix, defining an inner
product
$$
(v,w)_{P_x} = v^T P_x w,
$$
such that the associated norm $\norm{w}_{P_x} = (w,w)_{P_x}^{1/2}$
is equivalent to the norm $\norm{w}=(\Dx\sum_{k} w_k^2)^{1/2}$. 
Furthermore, for $D_x$ to be a summation-by-parts operator we require
that
$$
Q_x + Q_x^T = R_N-L_N,
$$
where $R_N$ and $L_N$ are the $N\times N$ matrices: $\diag{0,\dots,1}$
and $\diag{1,\dots,0}$ respectively.  Similarly, we can define a
summation-by-parts operator $D_y=P_y^{-1}Q_y$ approximating
$\partial_y$. Later we will also need the following Lemma, proven in \cite{svardsid3}.
\begin{lemma}\label{lemma:variable}
Given any smooth function $\bar{u}(x,y)$, we denote its restriction to the grid as $u$ and let $w$ be a smooth grid function. Then
\begin{equation}
 \label{eq:circdiff}
 \begin{aligned}
   \norm{D_x(u\circ w)- u\circ D_x w}_{P_x} \le C \norm{\partial_x \bar{u}}_{L^\infty([0,1])} \norm{w}_{P_x}
 \end{aligned}
\end{equation}
where $(u\circ v)_j=u_jv_j$.
\end{lemma}

Next, we move on to the two-dimensional case and discretize the unit square $[0,1]^2$ using $NM$ uniformly
distributed grid points $(x_i,y_j)=(i\Dx,j\Dy)$ for $i=0,\dots,N-1$,
and $j=0,\dots,M-1$, such that $(N-1)\Dx=(M-1)\Dy=1$. We order a scalar
grid function $w(x_i,y_i)=w_{ij}$ as a column vector
$$
w=\left(w_{0,0},w_{0,1},\dots,
  w_{0,(M-1)},w_{1,0},\dots,\dots,w_{(N-1),(M-1)}\right)^T.\label{eq:w}
$$ 

To obtain a compact notation for partial derivatives of a grid
function, we use Kronecker products. The Kronecker product of an
$N_1\times N_2$ matrix $A$ and an $M_1\times M_2$ matrix $B$ is
defined as the $N_1M_1\times N_2M_2$ matrix
\begin{equation}
\label{eq:order}
A\otimes B=
\begin{pmatrix}
  a_{11}B &\dots & a_{1N_2}B\\
  \vdots  & \ddots & \vdots \\
  a_{N_1 1}B  & \dots& a_{N_1 N_2}B  
\end{pmatrix}.
\end{equation}
For appropriate matrices $A$, $B$, $C$ and $D$, the Kronecker product
obeys the following rules:
\begin{align}
  (A\otimes B)(C\otimes D)&=(AC\otimes BD),\\
 (A\otimes B)+(C\otimes
  D)&=(A+C)\otimes(B+D),\\
   (A\otimes B)^T&=(A^T\otimes B^T).
\end{align}

Using Kronecker products, we can define 2-D difference operators. Let $I_n$ denote the $n\times n$ identity matrix, and define
$$
\bD_x = D_x \otimes I_M, \quad \bD_y = I_N \otimes D_y.
$$
For a smooth function $w(x,y)$, $(\bD_x w)_{i,j}\approx \partial_x
w(x_i,y_j)$ and similarly $(\bD_y w)_{i,j}\approx \partial_y
w(x_i,y_j)$. 

Set $\bP=P_x\otimes P_y$, define $(w,v)_{\bP}=w^T\bP v$ and the
corresponding norm $\norm{w}_{\bP}=(w,w)_{\bP}^{1/2}$. Also define 
$\mR=R_N\otimes I_M$, $\mL=L_N\otimes I_M$, $\mU=I_N \otimes R_M$ and 
$\mD=I_N \otimes L_M$. 

For a vector valued grid function $\mathbf{V}=(V^1,V^2)$, we use the following notation
$$
\bD_x \mathbf{V}=
\begin{pmatrix}
  \bD_x V^1\\ \bD_x V^2
\end{pmatrix},
$$
and so on. In the same spirit, the $\bP$ inner product of vector valued
grid functions is defined by $(\mathbf{V},\mathbf{W})_{\bP} = (V^1,W^1)_{\bP} + 
(V^2,W^2)_{\bP}$.

\begin{remark}
Note that the Kronecker products is just a tool to facilitate the
notation. In the implementation of schemes using the operators in the
Kronecker products we can think of these as operating in their own
dimension, i.e., on a specific index. Thus, to compute $\bD_x w$, we
can view $w$ as a field with two indices, and the one-dimensional
operator $D_x$ will operate on the first index since it appears in the
first position in the Kronecker product. 
\end{remark}

The usefulness of summation by parts operators comes from this lemma.
\begin{lemma}
  \label{lem:byparts} 
  For any grid functions $v$ and $w$, we have
  \begin{equation}
    \label{eq:byparts}
    \begin{aligned}
      \left(v,\bD_xw\right)_{\bP} + \left(\bD_xv,w\right)_{\bP} &= 
      v^T \left[(\mR - \mL)(I_N \otimes P_y)\right] w\\ 
      \left(v,\bD_yw\right)_{\bP} + \left(\bD_yv,w\right)_{\bP} &= 
      v^T \left[(\mU - \mD)(P_x \otimes I_M)\right] w.
    \end{aligned}
  \end{equation}
\end{lemma}
Observe that this lemma is the discrete version of the equality
$$
\iint_\Omega v\left(\partial_x w\right)\,dxdy + \iint_\Omega
\left(\partial_x v\right) w \,dxdy =   
\int_0^1 v(1,y) w(1,y) - v(0,y)w(0,y) \,dy.
$$
\begin{proof}
  We calculate
  \begin{align*}
    \left(v,\bD_xw\right) &= v^T\left(P_x\otimes P_y\right) 
    \left(P_x^{-1}Q_x\otimes I_M\right) w \\
    &=v^T Q_x \otimes P_y w \\
    &= -v^T Q_x^T \otimes P_y w + v^T (Q_x + Q_x^T) \otimes P_y w \\
    &= - (P_x^{-1}Q_x\otimes I_M v)^T \left(P_x\otimes P_y \right) w +
    v^T\left(R_N - L_N\right)\otimes P_y w\\
    &= - \left(\bD_xv\right)^T \left(P_x\otimes P_y \right) w 
    + v^T \left(\mR-\mL\right)(I_N \otimes P_y) w.
  \end{align*}
  The second equality is proved similarly.
\end{proof}

For a vector valued grid function $\bV$, we define the discrete
analogues of the $\Curl$ and the $\Curl(\Curl)$ operators by
$$
\begin{gathered}
  \crl\left( \bV\right) = \bD_x V^2 - \bD_y V^1 \ \text{and}\\
  \crl^2 \left(\bV\right) =
  \begin{pmatrix}
    -\bD_{yy} & \bD_{xy} \\
    \bD_{xy} &-\bD_{xx}
  \end{pmatrix}
  \begin{pmatrix}
    V^1\\ V^2
  \end{pmatrix}
  =
  \begin{pmatrix}
    -\bD_{yy} V^1 + \bD_{xy}V^2\\
    \bD_{xy} V^1 - \bD_{xx} V^2
  \end{pmatrix},
\end{gathered}
$$
where $\bD_{xy}=\bD_x\bD_y = \bD_y\bD_x$ etc.
Before we define our numerical schemes, we collect some useful results
in a lemma.
\begin{lemma} 
  \label{lem:kurlkurl} 
  \begin{equation}
    \label{eq:curlcurl}
    \begin{aligned}
      \left(\bV,\crl^2\left(\bV\right)\right)_{\bP}&=
      \norm{\crl\left(\bV\right)}_{\bP}^2 \\
      &\qquad + \left(V^1\right)^T\left[(\mU-\mD)\left(P_x\otimes
          I_M\right)\right] \crl\left(\bV\right) \\
      &\qquad - \left(V^2\right)^T
      \left[ (\mR-\mL)\left( I_N\otimes P_y\right)\right]
      \crl\left(\bV\right).
    \end{aligned}
  \end{equation}
  If $u$ is a grid function, then 
  \begin{equation}
    \label{eq:boundaries}
    \begin{aligned}
      \left(\bV,u\circ\bD_x\bV\right)_{\bP} &= \frac{1}{2} \bV^T
      [(\mR-\mL)(I_N\otimes
      P_y)] \left(u\circ\bV\right)
      \\
      &\qquad +
      \frac{1}{2}\left(u\circ \bD_x \bV - \bD_x\left(u\circ \bV\right),
        \bV\right)_{\bP},
      \\
      \left(\bV,u\circ\bD_y\bV\right)_{\bP} &= \frac{1}{2} \bV^T
      [(\mU-\mD)(P_x\otimes I_M)] \left(u\circ\bV\right)
      \\
      &\qquad +
       \frac{1}{2}\left(u\circ \bD_y \bV - \bD_y\left(u\circ \bV\right),
        \bV\right)_{\bP}.
    \end{aligned}
  \end{equation}
\end{lemma}
\begin{proof}
  To prove \eqref{eq:curlcurl}, 
  \begin{align*}
     \left(\bV,\crl^2\left(\bV\right)\right)_{\bP} &=
     \left(V^1,-\bD_{yy}V^1 + \bD_{yx}V^2\right)_{\bP} + 
     \left(V^2,\bD_{xy}V^1 - \bD_{xx}V^2\right)_{\bP} \\
     &= -\left(\bD_yV^1,-\bD_{y}V^1 + \bD_{x}V^2\right)_{\bP} -
     \left(\bD_xV^2,\bD_{y}V^1 - \bD_{x}V^2\right)_{\bP} \\
     &\qquad + \left(V^1\right)^T\left[(\mU-\mD)\left(P_x\otimes
         I_M\right)\right] \left(-\bD_{y}V^1 + \bD_{x}V^2\right) \\
     &\qquad + \left(V^2\right)^T \left[ (\mR-\mL)\left(
       I_N\otimes P_y\right)\right]\left( \bD_{y}V^1 - \bD_{x}V^2\right)\\
     &=\left(\crl\left(\bV\right),\crl\left(\bV\right)\right)_{\bP}
     + \left(V^1\right)^T\left[(\mU-\mD)\left(P_x\otimes
         I_M\right)\right] \crl\left(\bV\right)
     \\&\qquad  - 
     \left(V^2\right)^T \left[ (\mR-\mL)\left(
       I_N\otimes P_y\right)\right] \crl\left(\bV\right).
  \end{align*}
  To show \eqref{eq:boundaries}, first note that since $\bP$ is diagonal, $(u\circ
  \bD \bV,\bV)_{\bP} = (\bD \bV,u\circ \bV)_{\bP}$. We use
  Lemma~\ref{lem:byparts} to calculate
  \begin{align*}
    \left(u\circ \bD_x \bV, \bV \right)_{\bP}
    &= \left(\bD_x(u\circ \bV\right),\bV)_{\bP} + 
    \left(u\circ \bD_x \bV - \bD_x\left(u\circ \bV\right), \bV\right)_{\bP} \\
    &= -\left(u\circ \bV, \bD_x \bV\right)_{\bP} + 
    \bV^T \left(\mR-\mL\right)\left(I_N\otimes P_y\right) (u\circ \bV)
    \\
    &\qquad +
    \left(u\circ \bD_x \bV - \bD_x\left(u\circ \bV\right), \bV\right)_{\bP}
    \\
     &= -\left(u\circ \bD_x \bV,\bV\right)_{\bP} + 
    \bV^T \left(\mR-\mL\right)\left(I_N\otimes P_y\right) (u\circ \bV)
    \\
    &\qquad +
    \left(u\circ \bD_x \bV - \bD_x\left(u\circ \bV\right), \bV\right)_{\bP}.
  \end{align*}
  This shows the first equation in \eqref{eq:boundaries}, the second
  is proved similarly.
\end{proof}
Now we are in a position to state our scheme(s). For $\ell=1$ or $2$
we will use the notation $u^\ell$ for both the grid function defined
by the function $u^\ell(x,y)$ and for the function itself. Similarly,
for the boundary values, we use the notation $h$ and $g$ for both
discrete and continuously defined functions. Hopefully,
it will be apparent from the context what we refer to. 

The differential equation \eqref{eq:viscousinduc2} will be discretized in an obvious manner. We incorporate the boundary conditions by penalizing boundary
values away from the desired ones with a $\mathcal{O}(1/\Dx)$ term. To
this end set 
$$
\mathcal{B}=\left[ \left(P_x^{-1}\otimes I_M\right)\left(\Sigma_{\mL}\mL +
    \Sigma_{\mR}\mR\right) + \left(I_N\otimes P_y^{-1}\right)
  \left(\Sigma_{\mD}\mD + \Sigma_{\mU}\mU\right)\right],
$$
where $\Sigma_{\mL}$, $\Sigma_{\mR}$, $\Sigma_{\mD}$ and
$\Sigma_{\mU}$ are diagonal matrices, with components $(\sigma_{\mL})_{jj}$  ordered in the same way as in (\eqref{eq:order}) (and similarly for the other penalty matrices), to be specified later. Furthermore, the following form of the penalty paramaters will be convenient:
\begin{align}
\mathcal{B}=\mathcal{B'}+\epsilon\mathcal{B''}
\end{align}
and similarly for $\Sigma_{\mL}$, $\sigma_{\mL}$ etc.
   
With this  notation the scheme for the differential equation
\eqref{eq:main} with boundary values $\B(x,y)=\mathbf{g}(x,y)$ reads
\begin{equation}
  \label{eq:dirichletdiscrete}
  \bV_t + u^1\circ \bD_x \bV + u^2\circ \bD_y \bV - C \bV + 
  \epsilon \crl^2(\bV) = \mathcal{B}(\bV-\mathbf{g}), \quad t>0,
\end{equation}
while $\bV(0)$ is given. Here $C$ denotes the matrix
$$
C=
\begin{pmatrix}
  -\bD_y u^2 & \bD_y u^1 \\ \bD_x u^2 & -\bD_x u^1 
\end{pmatrix}.
$$

\begin{theorem}
  \label{thm:stab1} Let $\bV$ be as solution to
  \eqref{eq:dirichletdiscrete} with $\mathbf{g}=0$. If the constants
  in $\mathcal{B}$ is chosen as
\begin{equation}
  \begin{gathered}
    (\sigma'_{\mR})_{N-1,j}\le \frac{u^{1,-}(1,y_j)}{2}, \ (\sigma'_{\mL})_{0,j} \le  -\frac{u^{1,+}(0,y_j)}{2},
    \ (\sigma'_{\mU})_{i,M-1}\le \frac{u^{2,-}(x_i,1)}{2}, \\
    \text{and}\ (\sigma'_{\mD})_{i,0} \le -\frac{u^{2,+}(x_i,0)}{2},
  \end{gathered}
  \label{eq:primeconstants}
\end{equation}
\begin{equation}
  \label{eq:doubleprimeconstants}
  \begin{gathered}
    \sigma''_{\mR}\le-\frac{1}{2p \,\Dx}, \ \sigma''_{\mL} \le -\frac{1}{2p \,\Dx}, 
    \ \sigma''_{\mU}\le-\frac{1}{2p\,\Dy}, \
    \text{and}\ \ \sigma''_{\mD} \le -\frac{1}{2p\,\Dy},
  \end{gathered}
\end{equation}
and all other entries are 0, then 
  \begin{equation}
    \label{eq:stabdirichlet}
    \norm{\bV(t)}^2_{\bP} \le e^{ct}\norm{\bV(0)}_{\bP}^2,  
  \end{equation}
  where $ u^{l,+} = (u^l \vee 0)$, $ u^{l,-} = (u^l \wedge 0)$, for $ l = 1,2$ and $c$ is a constant depending on $u^1$, $u^2$, and their derivative approximations, but not on $N$ or $M$. By construction of the SBP operators $p= p^x_0=p^x_{N-1}=p^y_0=p^y_{M-1}$, where $P_x=\Dx\,\diag{p^x_0,\dots,p^x_{N-1}}$ and 
$P^y=\Dy\,\diag{p^y_0,\dots,p^y_{M-1}}$. 
\end{theorem}
\begin{proof}
  Set $E(t)=(\bV,\bV)_{\bP}$. Taking the $\bP$ inner product of
  \eqref{eq:dirichletdiscrete} and $\bV$, we get 
  \begin{align*}
    \frac{1}{2}\frac{d}{dt} E + \epsilon\left(\bV,\crl^2(\bV)\right)_{\bP} &=
    -\left(\bV,u^1\circ\bD_x \bV\right)_{\bP} -
    \left(\bV,u^2\circ\bD_y \bV\right)_{\bP}\\
    &\qquad +
    \left(\bV,C\bV\right)_{\bP} +
    \left(\bV,\mathcal{B}\bV\right)_{\bP}.
   \end{align*}
   Using Lemma~\ref{lem:kurlkurl} we get
   \begin{align*}
    \frac{1}{2}\frac{d}{dt}E &+ \epsilon
    \norm{\crl\left(\bV\right)}_{\bP}^2\\
    &= -\epsilon\left(V^1\right)^T\left[(\mU-\mD)\left(P_x\otimes
        I_M\right)\right] \crl\left(\bV\right) +
    \epsilon\left(V^2\right)^T \left[ (\mR-\mL)\left( I_N\otimes
        P_y\right)\right] \crl\left(\bV\right)
    \\
    &\qquad -\left(\bV,u^1\circ\bD_x \bV\right)_{\bP} -
    \left(\bV,u^2\circ\bD_y \bV\right)_{\bP} +
    \left(\bV,C\bV\right)_{\bP} +
    \left(\bV,\mathcal{B}\bV\right)_{\bP}\\
    &= -\epsilon\left(V^1\right)^T\left[(\mU-\mD)\left(P_x\otimes
        I_M\right)\right] \crl\left(\bV\right) +
    \epsilon\left(V^2\right)^T \left[ (\mR-\mL)\left( I_N\otimes
        P_y\right)\right] \crl\left(\bV\right)
    \\
    &\qquad - \frac{1}{2} \bV^T\left[ (\mR-\mL)(I_N\otimes
      P_y)\right](u^1\circ \bV) - \frac{1}{2} \bV^T\left[
      (\mU-\mD)(P_x\otimes I_M)\right](u^2\circ \bV)
    \\
    &\qquad - \frac{1}{2}\left(u^1\circ \bD_x \bV - \bD_x(u^1\circ
      \bV),\bV\right)_{\bP} - \frac{1}{2}\left(u^2\circ \bD_y \bV -
      \bD_y(u^2\circ \bV),\bV\right)_{\bP}
    \\
    &\qquad + \left(\bV,C\bV\right)_{\bP} +
    \left(\bV,(\mathcal{B'}+\epsilon \mathcal{B''})\bV\right)_{\bP}.
  \end{align*}
Note that by \eqref{eq:circdiff}, 
\begin{equation}\label{eq:squareterms}
 \begin{aligned}
   \abs{ \left(u^1\circ \bD_x \bV - \bD_x(u^1\circ
       \bV),\bV\right)_{\bP}}&\le  c \norm{\bV}_{\bP}^2,\\
   \abs{\left(u^2\circ \bD_y \bV - \bD_y(u^2\circ
       \bV),\bV\right)_{\bP}}&\le  c \norm{\bV}_{\bP}^2,\\
   \abs{\left(\bV,C\bV\right)_{\bP}}&\le c \norm{\bV}_{\bP}^2,
 \end{aligned}
\end{equation}
for some constant $c$ depending on the first derivatives of $u^1$
and $u^2$.  Using the conditions (\ref{eq:primeconstants}) we arrive at
   \begin{align*}
    \frac{1}{2}\frac{d}{dt}E &+ \epsilon
    \norm{\crl\left(\bV\right)}_{\bP}^2\\
    &\leq cE-\epsilon\left(V^1\right)^T\left[(\mU-\mD)\left(P_x\otimes
       I_M\right)\right] \crl\left(\bV\right) +
   \epsilon\left(V^2\right)^T \left[ (\mR-\mL)\left( I_N\otimes
       P_y\right)\right] \crl\left(\bV\right)
   \\
    &\quad 
    + \epsilon\left(\bV,\mathcal{B''}\bV\right)_{\bP}.
  \end{align*}
Next, for any grid function $w$ (with components as in (\eqref{eq:circdiff})), we have
   \begin{align*}
     \norm{w}^2_{\bP} &= \sum_{j=0}^{M-1} \Dy\, p^y_j \sum_{i=0}^{N-1} \Dx\, p^x_i
     w^2_{i,j}\\
     &\ge \Dy \, p \sum_{i=0}^{N-1} \Dx \, p^x_i 
     \left(w^2_{i,0} + w^2_{i,M-1}\right).
   \end{align*}
   Similarly 
   $$
   \norm{w}^2_{\bP}\ge \Dx\, p \sum_{j=0}^{M-1} \Dy\,p^y_j  
   \left(w^2_{0,j}+w^2_{(N-1),j}\right).
   $$
   Combining this we find that 
   \begin{equation}
     \begin{aligned}
       \norm{w}^2_{\bP} &\ge \frac{\Dx \,p}{2}w^T \Bigl( (I_N\otimes P_y)\mL +
       (I_N\otimes P_y)\mR\Bigr)w \\
       &+\frac{\Dy\, p}{2}w^T\Bigl( ((P_x\otimes I_M)\mD + (P_x\otimes
       I_M)\mU \Bigr)w.
     \end{aligned}\label{eq:wPnorm}
    \end{equation}
   We also compute
   \begin{equation}
   \begin{aligned}
     (w,\mathcal{B}w)_{\bP} &=
     w^T\Bigl(
        (I_N\otimes P_y)\Sigma_{\mL}\mL + (I_N\otimes
       P_y)\Sigma_{\mR}\mR\\
       &\hphantom{w^T\Bigl(}\quad +
       (P_x\otimes I_M)\Sigma_{\mD}\mD + (P_x\otimes I_M)\Sigma_{\mU}\mU\Bigr)w.
   \end{aligned}\label{eq:wBw}
 \end{equation}
 Using \eqref{eq:wPnorm} for $w=\crl(\bV)$, and \eqref{eq:wBw} for
 $w=\bV$ we find
 \begin{align*}
   &\frac{d}{dt}E \le cE \\
   & \qquad -\epsilon\Bigl(\frac{p_{dx}}{2}\crl(\bV)^T \bigl((I_N\otimes
   P_y)\mL + (I_N\otimes P_y)\mR \bigr) \crl(\bV)\\
   & \qquad \quad + \frac{p_{dy}}{2}\crl(\bV)^T \bigl(
   (P_x\otimes I_M)\mD + (P_x\otimes I_M)\mU\bigr)\crl(\bV)\\
   &\quad \hphantom{-\epsilon\Bigl(} +
   \left(V^1\right)^T\left[(\mU-\mD)\left(P_x\otimes I_M\right)\right]
   \crl\left(\bV\right) - \left(V^2\right)^T \left[ (\mR-\mL)\left(
       I_N\otimes
       P_y\right)\right] \crl\left(\bV\right) \\
   &\quad \hphantom{-\epsilon\Bigl(} + \bV^T \bigl( -
   \sigma''_{\mL}(I_N\otimes P_y)\mL - \sigma''_{\mR} (I_N\otimes P_y)\mR
   -\sigma''_{\mD} (P_x\otimes I_M)\mD - \sigma''_{\mU} (P_x\otimes
   I_M)\mU\bigr)
   \bV\Bigr)\\ 
   &
   =:-\epsilon A
\end{align*}
Choose the remaining penalty parameters as in (\ref{eq:doubleprimeconstants}) and write
\begin{align*}
   A&\ge \frac{p_{dx}}{2}\left\langle \crl(\bV),\crl(\bV)\right\rangle_{\mL} + 
   \left\langle V^2,\crl(\bV)\right\rangle_{\mL} - 
   \sigma''_{\mL}\left\langle V^2,V^2 \right\rangle_{\mL}\\
  &\quad + 
    \frac{p_{dx}}{2}\left\langle \crl(\bV),\crl(\bV)\right\rangle_{\mR} - 
   \left\langle V^2,\crl(\bV)\right\rangle_{\mR} - 
   \sigma''_{\mR}\left\langle V^2,V^2 \right\rangle_{\mR} \\
  &\quad
   +  \frac{p_{dy}}{2}\left\langle \crl(\bV),\crl(\bV)\right\rangle_{\mD} -
   \left\langle V^1,\crl(\bV)\right\rangle_{\mD} - 
   \sigma''_{\mD}\left\langle V^1,V^1 \right\rangle_{\mD} \\
   &\quad 
   +  \frac{p_{dy}}{2}\left\langle \crl(\bV),\crl(\bV)\right\rangle_{\mU} +
   \left\langle V^1,\crl(\bV)\right\rangle_{\mU} - 
   \sigma''_{\mU}\left\langle V^1,V^1 \right\rangle_{\mU}\\
   &= \frac{1}{2}\left\langle \frac{1}{\sqrt{p_{dx}}}V^2+\sqrt{p_{dx}}\crl(\bV),
      \frac{1}{\sqrt{p_{dx}}}V^2+\sqrt{p_{dx}}\crl(\bV)\right\rangle_{\mL} +
    \left(-\sigma''_{\mL}-\frac{1}{2p\Dx}\right)\left\langle
      V^2,V^2\right\rangle_{\mL}\\
    &\quad + 
    \frac{1}{2}\left\langle \frac{1}{\sqrt{p_{dx}}}V^2-\sqrt{p_{dx}}\crl(\bV),
      \frac{1}{\sqrt{p_{dx}}}V^2-\sqrt{p_{dx}}\crl(\bV)\right\rangle_{\mR} +
    \left(-\sigma''_{\mR}-\frac{1}{2p\Dx}\right)\left\langle
      V^2,V^2\right\rangle_{\mR}
    \\
    &\quad + 
    \frac{1}{2}\left\langle \frac{1}{\sqrt{p_{dy}}}V^1-\sqrt{p_{dy}}\crl(\bV),
      \frac{1}{\sqrt{p_{dy}}}V^1-\sqrt{p_{dy}}\crl(\bV)\right\rangle_{\mD} +
    \left(-\sigma''_{\mD}-\frac{1}{2p\Dy}\right)\left\langle
      V^1,V^1\right\rangle_{\mD}\\
    &\quad + 
    \frac{1}{2}\left\langle \frac{1}{\sqrt{p_{dy}}}V^1+\sqrt{p_{dy}}\crl(\bV),
      \frac{1}{\sqrt{p_{dy}}}V^1+\sqrt{p_{dy}}\crl(\bV)\right\rangle_{\mU} +
    \left(-\sigma''_{\mU}-\frac{1}{2p\Dy}\right)\left\langle
      V^1,V^1\right\rangle_{\mU}\\
    &\ge 0,
 \end{align*}
where 
 $$
 \left\langle v,w \right\rangle_{\mR} = 
 v^T (I_N\otimes P_y) \mR w \ \text{and so on are positive semi-definite forms.}
 $$
Furthermore, we used the notation $p_{dx}=p\Dx$ and $p_{dy}=p\Dy$.
 Summing up, we have shown that
 $$
 \frac{d}{dt}E(t)\le cE(t),
 $$
 and the result follows by Gronwall's inequality.
\end{proof}
The scheme for the mixed boundary conditions reads
\begin{equation}
  \label{eq:mixedscheme}
  \begin{aligned}
    \bV_t + &u^1\circ \bD_x \bV + u^2\circ \bD_y \bV - C \bV + \epsilon
    \crl^2(\bV)\\
    &= \mathcal{B}(\bV-\mathbf{g}) +\epsilon
    \begin{pmatrix}
      (\mU-\mD)\left(I_N\otimes P_y^{-1}\right)(\crl(\bV)-\mathbf{h})\\
      -(\mR-\mL)\left(P_x^{-1}\otimes I_M\right)(\crl(\bV)-\mathbf{h})
    \end{pmatrix},
    \quad t>0,
  \end{aligned}
\end{equation}
where $\mathbf{h}$ is the desired value of $\crl(\bV)$ on the
boundary.
\begin{theorem}
  \label{thm:mixedtype} If $\bV$ is a solution of
  \eqref{eq:mixedscheme} with $\mathbf{h}=\mathbf{g}=0$, and
  $\mathcal{B}(=\mathcal{B}')$ is chosen so that \eqref{eq:primeconstants} holds,
  then
  \begin{equation}
    \label{eq:mixedstable}
    \norm{\bV(t)}^2_{\bP} + \epsilon\int_0^t e^{c(t-s)} 
    \norm{\crl(\bV(s))}^2_{\bP}\,ds \le  e^{ct}\norm{\bV(0)}^2_{\bP},
  \end{equation}
  where $c$ is a constant depending on $\bD_x$, $\bD_y$ and the
  derivatives of $u^1$ and $u^2$, but not on $N$ or $M$. 
\end{theorem}
\begin{proof}
  The proof of this theorem proceeds as the proof of
  Theorem~\ref{thm:stab1}. Note that we are subtracting the boundary
  terms coming from $(\bV,\crl^2(\bV))_{\bP}$, so that we do not need
  to split $\mathcal{B}$. The other terms are estimated as before, and
  we get the inequality
  $$
  \frac{d}{dt} E + \epsilon\norm{\crl(\bV)}^2_{\bP} \le cE,
  $$
  which yields the stability result.
\end{proof}
\begin{remark}
We have assumed a constant resistivity co-efficient $\eps$ in the above discussion. However, in many practical applications, the co-efficient of resistivity can vary in space. In such cases, our theoretical results hold provided that the resistivity co-efficient is uniformly bounded away from zero.
\end{remark}

The analysis has been carried out on a Cartesian equidistant grid on the unit square. However, this is not a restriction as problems on general domains may be addressed using coordinate transformations. The stability proofs will hold as long as the norm matrices ($P_x.P_y,P_z$) are diagonal. SBP finite-difference schemes with diagonal $P$ matrix have a truncation error of $2p$ in the interior and $p$ near the boundary resulting in a global order of accuracy/convergence rate of $p+1$. (See \cite{Svard} for further details.)

The resistive magnetic induction equations include diffusive terms. Those are discretized by applying the first-derivative operators twice. This results in a truncation error of $p-1$ near the boundary for the diffusive terms. However, thanks to the energy stability of the scheme, the global convergence rate and order of accuracy remains at $p+1$ (see \cite{SvardNordstrom}).

\section{Schemes in Three-dimensions}
\label{sec:3d}
In this section, we are going to write down the three-dimensional version of the finite difference scheme for the equation \eqref{eq:viscousinduc1}. To begin with, we discretize unit cube 
$[0,1]^3$ using uniformly distributed grid points $( x_{i}, y_{j}, z_{k})=(i\Dx,j\Dy,k\Dz)$ for $i=0,\dots,N-1$,
 $j=0,\dots,M-1$, and $ k= 0,\dots, K-1$ such that $(N-1)\Dx=(M-1)\Dy= (K-1)\Dz=1$. We order a scalar
grid function $w(x_i,y_i,z_k)=w_{ijk}$ as a column vector
\begin{align*}
w& =(w_{0,0,0},w_{0,0,1},\dots,w_{0,0,(K-1)},w_{0,1,0},w_{0,1,1},\dots,w_{0,(M-1),(K-1)},\dots,\\
& \qquad \qquad w_{1,0,0},\dots,\dots,w_{(N-1),(M-1),(K-1)})^T.
\end{align*}

As before, let $I_n$ denote the $n\times n$ identity matrix, and define
$$
\bD_x = D_x \otimes I_M \otimes I_K, \quad \bD_y = I_N \otimes D_y \otimes I_K, \quad \bD_z = I_N \otimes I_M \otimes D_z.
$$
Set $\bP=P_x\otimes P_y\otimes P_z$, define $(w,v)_{\bP}=w^T\bP v$ and the
corresponding norm $\norm{w}_{\bP}=(w,w)_{\bP}^{1/2}$. Also define 
$\mR=R_N\otimes I_M \otimes I_K$, $\mL=L_N\otimes I_M \otimes I_K$, $\mU=I_N \otimes R_M \otimes I_K$ and 
$\mD=I_N \otimes L_M \otimes I_K$, $\mA=I_N \otimes I_M \otimes R_K $, $\mB=I_N \otimes I_M \otimes L_K $. 

For a vector valued grid function $\mathbf{V}=(V^1,V^2,V^3)$, we use the following notation
$$
\bD_x \mathbf{V}=
\begin{pmatrix}
  \bD_x V^1\\ \bD_x V^2\\ \bD_x V^3
\end{pmatrix},
$$
and so on. In the same spirit, the $\bP$ inner product of vector valued
grid functions is defined by $(\mathbf{V},\mathbf{W})_{\bP} = (V^1,W^1)_{\bP} + 
(V^2,W^2)_{\bP} + (V^3,W^3)_{\bP}$.
Finally, we set 
\begin{align*}
\mathcal{S}&=[\left(P_x^{-1}\otimes I_M \otimes I_K\right)\left(\Sigma_{\mL}\mL +
    \Sigma_{\mR}\mR\right) + \left(I_N\otimes P_y^{-1}\otimes I_K\right)
  \left(\Sigma_{\mD}\mD + \Sigma_{\mU}\mU\right)\\
 & \qquad \qquad + \left(I_N\otimes I_M \otimes P_z^{-1}\right)
  \left(\Sigma_{\mA}\mA + \Sigma_{\mB}\mB\right)],
\end{align*}
where $\Sigma_{\mL}$, $\Sigma_{\mR}$, $\Sigma_{\mD}$,$\Sigma_{\mU}$,$\Sigma_{\mA}$ and
$\Sigma_{\mB}$ are diagonal matrices.
With these notations above the scheme for the differential equation \eqref{eq:viscousinduc1}
with boundary values $\B(x,y,z)=\mathbf{g}(x,y,z)$ reads
\begin{equation}
  \label{eq:dirichletdiscrete3d}
  \bV_t + u^1\circ \bD_x \bV + u^2\circ \bD_y \bV + u^3\circ \bD_z \bV - C \bV + 
  \epsilon \crl^2(\bV) = \mathcal{S}(\bV-\mathbf{g}), \quad t>0,
\end{equation}
while $\bV(0)$ is given. Here $C$ denotes the matrix
$$
C=
\begin{pmatrix}
  -\bD_y u^2 - \bD_z u^3 & \bD_y u^1 & \bD_z u^1 \\ \bD_x u^2 & -\bD_x u^1 - \bD_z u^3 & \bD_z u^2 \\ \bD_x u^3 & \bD_y u^3 & -\bD_x u^1 -\bD_y u^2 
\end{pmatrix},
$$
and 
$$
\crl^2 \left(\bV\right) =
  \begin{pmatrix}
    \bD_{y}\left( \bD_x V^2 -\bD_y V^1 \right) - \bD_z \left( \bD_z V^1 -\bD_x V^3\right) \\
    \bD_{z}\left( \bD_y V^3 -\bD_z V^2 \right) - \bD_x \left( \bD_x V^2 -\bD_y V^1\right) \\
    \bD_{x}\left( \bD_z V^1 -\bD_x V^3 \right) - \bD_y \left( \bD_y V^3 -\bD_z V^2\right) 
  \end{pmatrix}.
$$

\begin{remark}
Note that the stability result for the three-dimensional scheme can be achieved along the same way as in Theorem~\ref{thm:stab1} by choosing the penalty parameters as
\begin{equation}
  \begin{gathered}
    (\sigma'_{\mR})_{N-1,j,k}\le \frac{u^{1,-}(1,y_j,z_k)}{2}, \ (\sigma'_{\mL})_{0,j,k} \le  -\frac{u^{1,+}(0,y_j,z_k)}{2},
    \ (\sigma'_{\mU})_{i,M-1,k}\le \frac{u^{2,-}(x_i,1,z_k)}{2}, \\
   (\sigma'_{\mD})_{i,0,k} \le -\frac{u^{2,+}(x_i,0,z_k)}{2}, \ (\sigma'_{\mA})_{i,j,K-1} \le \frac{u^{3,-}(x_i,y_j,1)}{2},
\text{and}\ (\sigma'_{\mB})_{i,j,0} \le -\frac{u^{3,+}(x_i,y_j,0)}{2},
  \end{gathered}
  \label{eq:3dprimeconstants}
\end{equation}
\begin{equation}
  \label{eq:3ddoubleprimeconstants}
  \begin{gathered}
    \sigma''_{\mR}\le-\frac{1}{2p \,\Dx}, \ \sigma''_{\mL} \le -\frac{1}{2p \,\Dx}, 
    \ \sigma''_{\mU}\le-\frac{1}{2p\,\Dy}, \
    \sigma''_{\mD} \le -\frac{1}{2p\,\Dy}, \\
   \sigma''_{\mA} \le -\frac{1}{2p\,\Dz},
   \text{and}\ \ \sigma''_{\mB} \le -\frac{1}{2p\,\Dz},
  \end{gathered}
\end{equation}
and all other entries are 0.
\end{remark}
\section{Numerical Experiments}
\label{sec:numex}

The SBP-SAT schemes \eqref{eq:dirichletdiscrete} have been tested on a suite of
numerical experiments in order to demonstrate their effectiveness. We have used the second-order (first-order) accurate and the
fourth-order (second-order) accurate SBP operators in the interior
(boundary). From the results of \cite{SvardNordstrom}, these operators result in
overall second and third order accurate discretizations of the
equations. Henceforth, the second (first)-order accurate SBP scheme
will be denoted as $SBP2$ and the fourth (second)-order accurate SBP
scheme will be denoted as $SBP4$ after their orders of accuracy in the
interior.  Time integration is performed by using a standard second-order
accurate Runge-Kutta scheme. (Using a higher-order Runge Kutta scheme didn't affect the quality of the computational results.)

\subsection*{Numerical experiment $1$:}
To begin with, we consider a divergence free velocity field $\U(x,y) = (-y,x)^T$ and a slightly modified form of \eqref{eq:main} given by 
\begin{equation}
  \label{eq:forceeqn}
  \begin{aligned}
    \B_t + \Lambda_1\B_x + \Lambda_2\B_y - C\B & = \epsilon \left[
        \begin{aligned}
          - ((B^2)_{xy}-(B^1)_{yy})\\
           ((B^2)_{xx}-(B^1)_{xy})
        \end{aligned}
      \right] + \mathcal{F},
    \end{aligned}
\end{equation}
where the forcing function $\mathcal{F}$ is given by,
\begin{equation}
\label{eq:force}
\begin{aligned}
&f_{1}= 160\epsilon (y-0.5\sin(t))
\left[
-4 + 40\lbrace(x-0.5 \cos(t))^2 + (y-0.5 \sin(t))^2\rbrace
\right]
e^{A(t)},\\
&f_{2}= - 160\epsilon (y-0.5\cos(t))
\left[
-4 + 40\lbrace(x-0.5 \cos(t))^2 + (y-0.5 \sin(t))^2\rbrace
\right]
e^{A(t)},
\end{aligned}
\end{equation}
with $A$: 
$$
A(t) = -20\lbrace(x\cos(t)+y\sin(t)-0.5)^2 + (-x\sin(t)+ y\cos(t))^2\rbrace
$$.

We note that it is straightforward to extend the stability results of the previous section to SBP-SAT schemes for \eqref{eq:forceeqn}. The forcing term is evaluated in a standard manner. The forcing function in \eqref{eq:forceeqn} enables us to calculate an exact (smooth) solution of the equation given by,
\begin{equation}
  \label{eq:ex}
  \B(\X,t)=R(t)\B_0(R(-t)\X),
\end{equation}
where $R(t)$ is a rotation matrix with angle $t$. Note that this exact solution represents the rotation of the initial data about the origin. In fact, \eqref{eq:ex} is also an exact solution of \eqref{eq:forceeqn} with both the forcing term $\mathcal{F}$ and the resistivity $\epsilon$ set to zero. Hence, this example follows from a similar example for the inviscid magnetic induction equations considered in \cite{TF1,fkrsid1}.

For initial data, we choose the divergence free magnetic field:
\begin{equation}
  \B_0(x,y)=4
  \begin{pmatrix}
    -y\\ x-\frac{1}{2}
  \end{pmatrix}
  e^{-20\left((x-1/2)^2+y^2\right)},
\end{equation}
and the computational domain $[-1,1] \times [-1,1]$. In this case, the
exact solution is a smooth hump (centered at $(1/2,0)$
and decaying exponentially) rotating about the origin and completing
one rotation in time $t = 2 \pi$. The hump remains completely inside the domain during the course of the rotation. Since the exact solution is known in
this case, we use this solution to specify the
data for the boundary conditions \eqref{eq:dirichletboundary} or \eqref{eq:mixedboundary}. 
The above setup is simulated using the SBP2 and SBP4 schemes. Using Dirichlet or mixed boundary conditions led to very similar results. Hence, we present results only with the mixed boundary conditions \eqref{eq:mixedboundary} in this case. The time-integration was performed with a second-order Runge-Kutta method at a $CFL$ number of $0.5$. The resistivity $\epsilon = 0.01$ was used. Also we have used the value of the penalty parameters as mentioned in Theorem~\ref{thm:stab1}, for instance we used 
$ \sigma_{\mL} = \sigma'_{\mL} + \epsilon \sigma''_{\mL} =  - \frac{u^{1,+}}{2} - \frac{\epsilon}{2p\Dx}$ and so on. We plot the $l^2$ norm of the magnetic field: $\B = \sqrt{(B^1)^2 + (B^2)^2}$, at times $t=\pi$ (half rotation) and $t=2\pi$ (full rotation) for both the SBP2 and SBP4 schemes in figure \ref{fig:1}.
\begin{figure}[htbp]
  \centering
    \subfigure[SBP2, half rotation]{\includegraphics[width=0.45\linewidth]{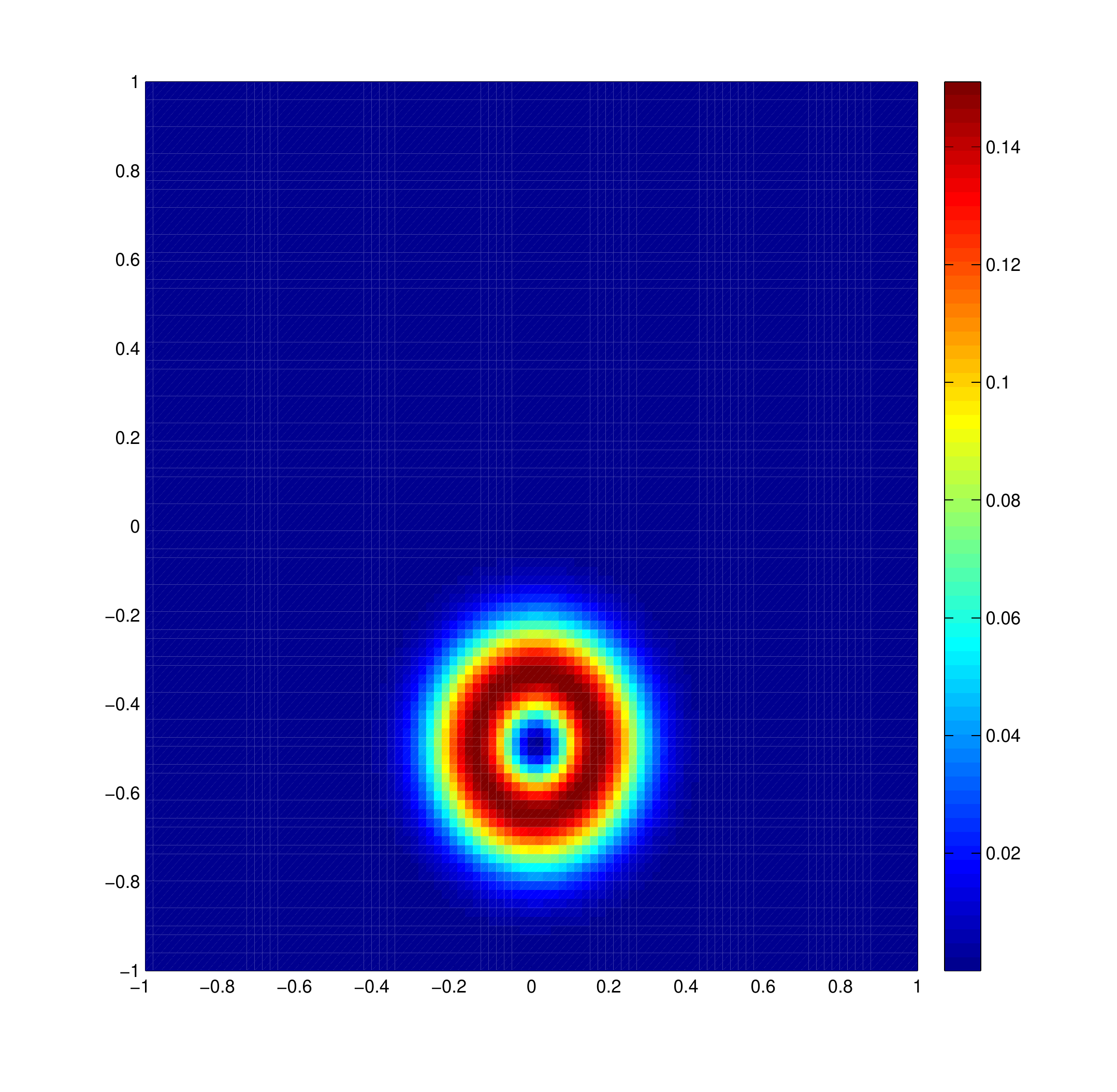}} 
    \subfigure[SBP2, full rotation]{\includegraphics[width=0.45\linewidth]{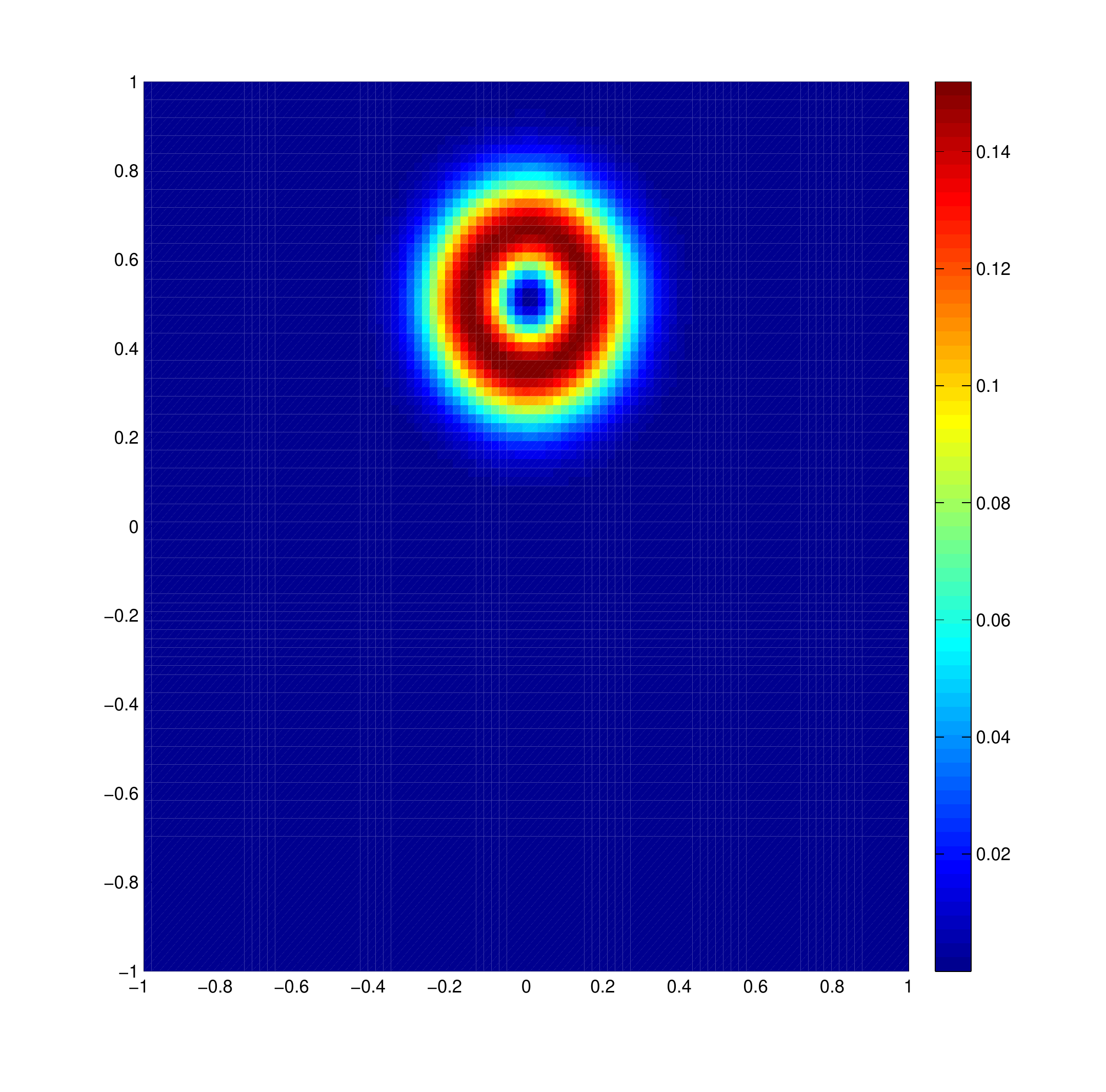}}
    \subfigure[SBP4, half rotation]{\includegraphics[width=0.45\linewidth]{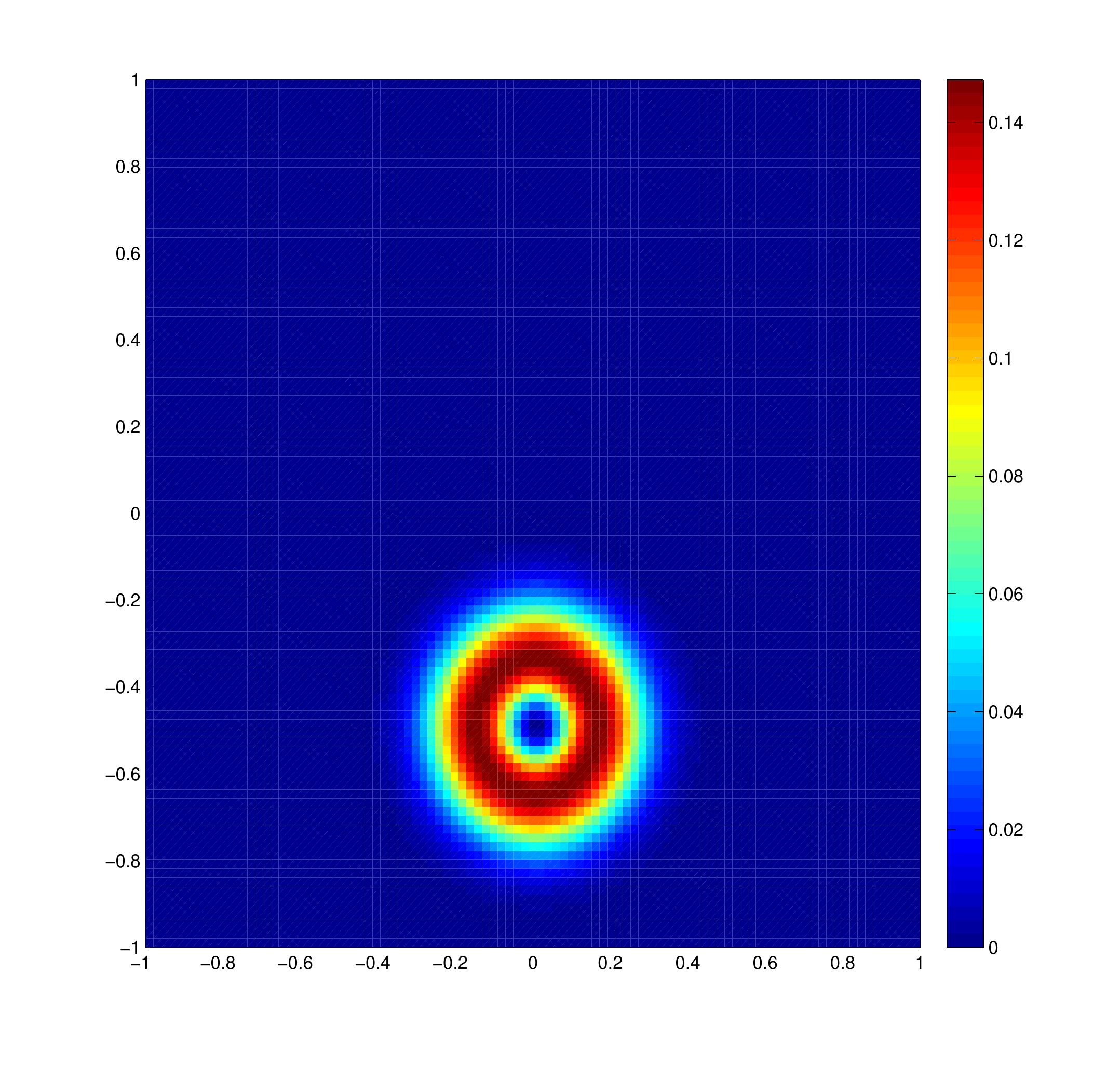}}
    \subfigure[SBP4, full rotation]{\includegraphics[width=0.45\linewidth]{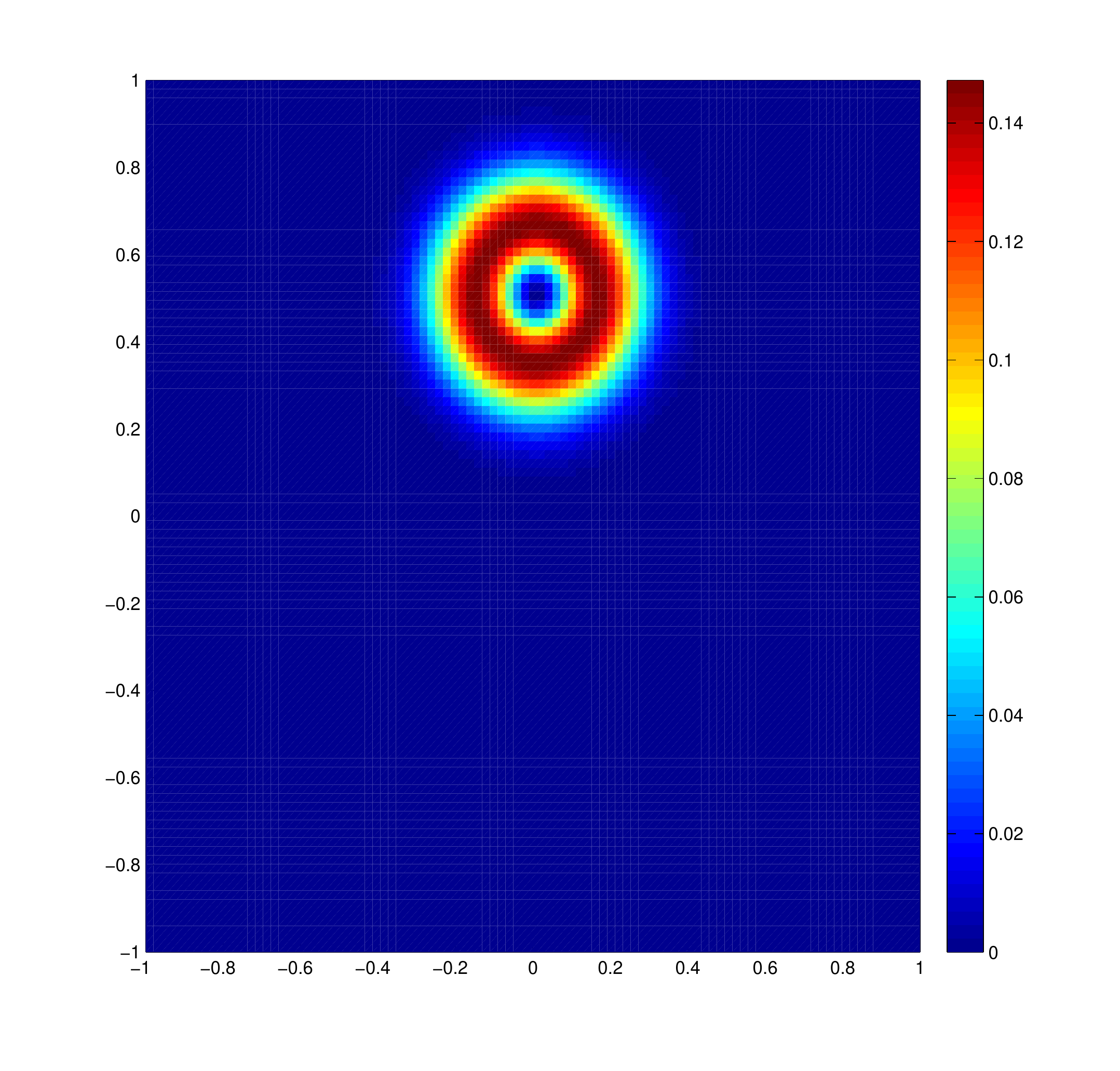}}
\caption{Numerical Experiment $1$: $\| \B \| = \sqrt{\B_1^2 + \B_2^2}$ for the SBP2 and SBP4 schemes on a $100 \times 100$ mesh.}
\protect \label{fig:1}
\end{figure}
As shown in the figure, both $SBP2$ and $SBP4$ schemes resolve the
solution quite well. There are very few noticeable differences between the second and fourth order schemes at this resolution. The shape of the hump is maintained during the rotation. A quantitative view of the
results is presented in Table~\ref{tab:1}
where
\begin{align}
\textrm{relative percentage error}=\frac{\|\B_{num}-\B_{ex}\|_2}{\|\B_{ex}\|_2}\times 100,
\end{align}
$\B_{num}$ is the numerical approximation and $\B_{ex}$ is the exact reference solution and $\norm{w}_{2}=(\Dx\sum_{k} w_k^2)^{1/2}$. 
\begin{table}[htbp]
\centering
\begin{tabular}{c|rrrr}
Grid size& $SBP2$ & rate & $SBP4$ & rate\\
\hline
40$\times$40   & 2.1e{-1}         &                      & 1.6e{-2}     &  \\
80$\times$80   & 5.7e{-2}         & 1.9                  & 1.1e{-3}     & 3.9 \\
160$\times$160 & 1.3e{-2}         & 2.1                  & 1.1e{-4}     & 3.3 \\
320$\times$320 & 3.1e{-3}         & 2.0                  & 1.3e{-5}     & 3.1 \\
640$\times$640 & 7.5e{-4}         & 2.0                  & 1.6e{-6}     & 3.0
\end{tabular}
\caption{Relative percentage errors in $L^2$ for $\| \B\|$ at time $t = 2
  \pi$ and rates of convergence for
  numerical experiment $1$.} 
\label{tab:1}
\end{table}
The errors are computed at time $t = 2 \pi$ (one rotation) on a sequence
of meshes with both the $SBP2$ and $SBP4$ schemes. The results show
that the errors are quite low, particularly for $SBP4$ and the
rate of convergence approaches $2$ for $SBP2$ and $3$ for $SBP4$. This
is consistent with the theoretical order of accuracy for SBP operators
(see \cite{SvardNordstrom}). The very low values of error with $SBP4$ suggest
that one should use high order schemes to resolve
interesting solution features. 

Another feature of numerical solutions of equations \eqref{eq:forceeqn} is the behavior of divergence of the magnetic field. Note that both the initial data and the forcing function \eqref{eq:force} are divergence free. Hence, the divergence of the exact solution of \eqref{eq:forceeqn} should remain zero for all time. However, as remarked before, we don't attempt to preserve any particular discrete form of divergence. Hence, numerical divergence errors are an indicator of the performance of the schemes. We define
the discrete divergence operator:
\begin{align*}
  \mathrm{div}_P(V)= \bD_xV^{1} + \bD_y V^{2}.
\end{align*}
This corresponds to the standard centered discrete divergence operator
at the corresponding orders of accuracy.  The divergence errors in
$l^2$ and rates of convergence at time $t=2\pi$ for the $SBP2$ and
$SBP4$ schemes on a sequence of meshes are presented in
Table~\ref{tab:2}.
\begin{table}[htbp]
\centering
\begin{tabular}{c|rrrr}
Grid size  & $SBP2$ & rate  & $SBP4$ & rate \\
\hline
20$\times$20  & 8.9e-1   &       & 4.5e-1   &  \\
40$\times$40  & 3.9e-1   & 1.1   & 8.0e-2   & 2.9 \\
80$\times$80  & 1.0e-1   & 2.0   & 4.2e-3   & 3.8 \\
160$\times$160& 2.7e-2   & 1.9   & 5.0e-4   & 3.0 \\
320$\times$320& 9.5e-3   & 1.5   & 8.0e-5   & 2.6
\end{tabular}
\caption{Numerical Experiment $1$:
  Divergence errors in $l^2$ and rates of convergence at time
  $t=2\pi$. }
\label{tab:2} 
\end{table}
From Table~\ref{tab:2}, we conclude that although the initial
divergence is zero, the discrete divergence computed with both the
$SBP2$ and $SBP4$ schemes is not zero. However, the divergence errors
are very low in magnitude even on fairly coarse meshes and converge to
zero at a rate of $1.5$ and $2.5$ for $SBP2$ and $SBP4$ scheme
respectively. A simple truncation error analysis suggests that these rates
for the $SBP2$ and $SBP4$ schemes are optimal. 
Finally, we emphasize again that the quality of the solutions are good and the convergence rates did not suffer, despite the scheme not preserving any form of discrete divergence.

\subsection*{Numerical Experiment $2$}
In the previous numerical experiment, the hump (representing the interesting
parts of the solution) was confined to the interior of the domain. A more challenging test of the boundary closures is provided if the hump interacts with the boundary. We proceed to test this situation by considering \eqref{eq:forceeqn} in a domain $[0,1] \times [0,1]$ with exactly the same initial data, resistivity and forcing function as in the previous numerical experiment. The exact solution \eqref{eq:ex}, being a rotation about the origin, now exits the domain first at the lower boundary (including a corner) and enters the domain through another part of the boundary during the course of a single rotation. We study this interaction by simulating \eqref{eq:forceeqn} with the SBP2 and SBP4 schemes. We consider \eqref{eq:forceeqn} with Dirichlet boundary conditions. The boundary data are calculated by evaluating the exact solution \eqref{eq:ex} at the boundary.

The $L^2$ norm of $\B$ after half a rotation and one full rotation is plotted in figure \ref{fig:2}. 
\begin{figure}[htbp]
  \centering
    \subfigure[SBP2, half rotation]{\includegraphics[width=0.45\linewidth]{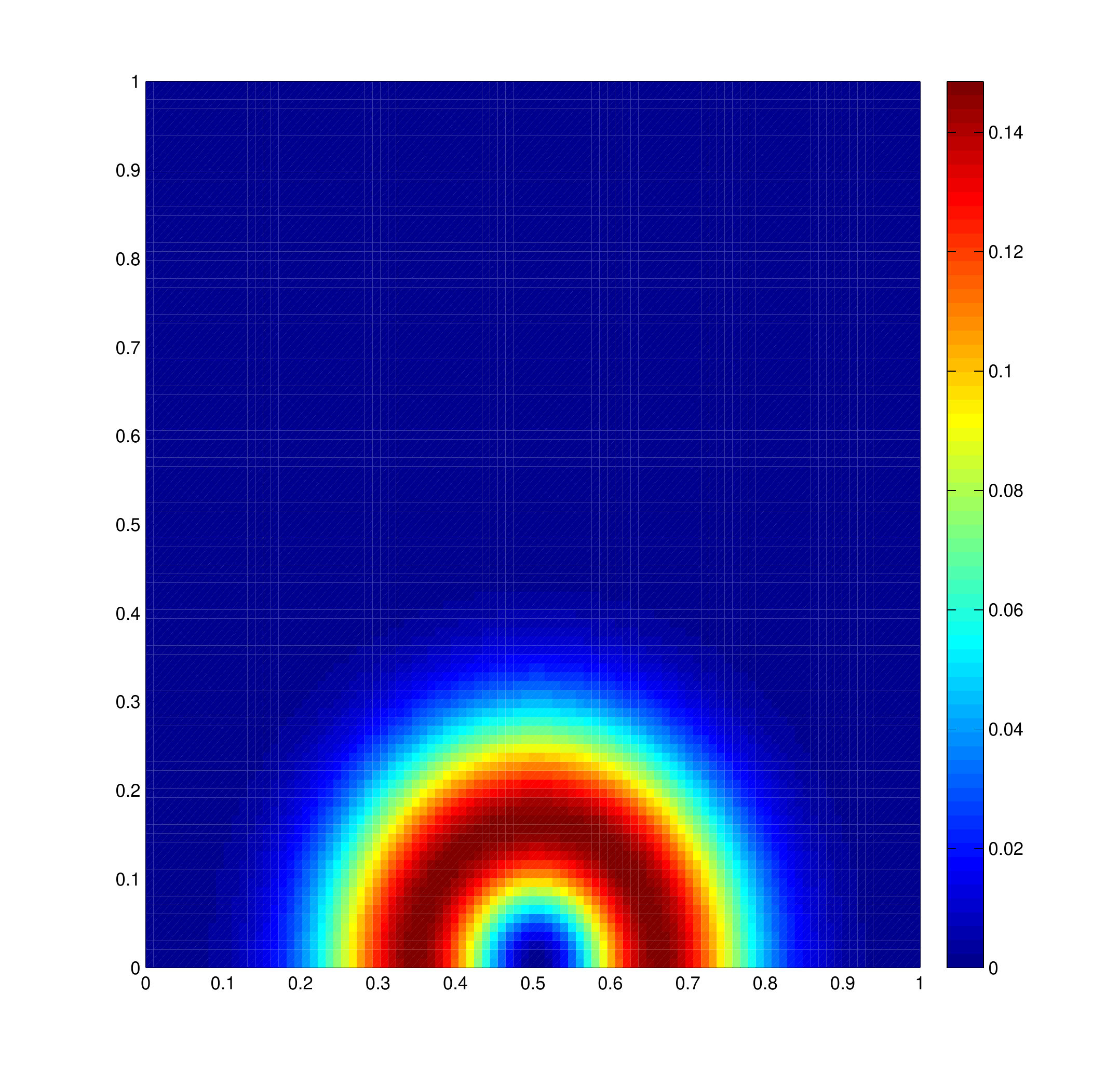}} 
    \subfigure[SBP2, full rotation]{\includegraphics[width=0.45\linewidth]{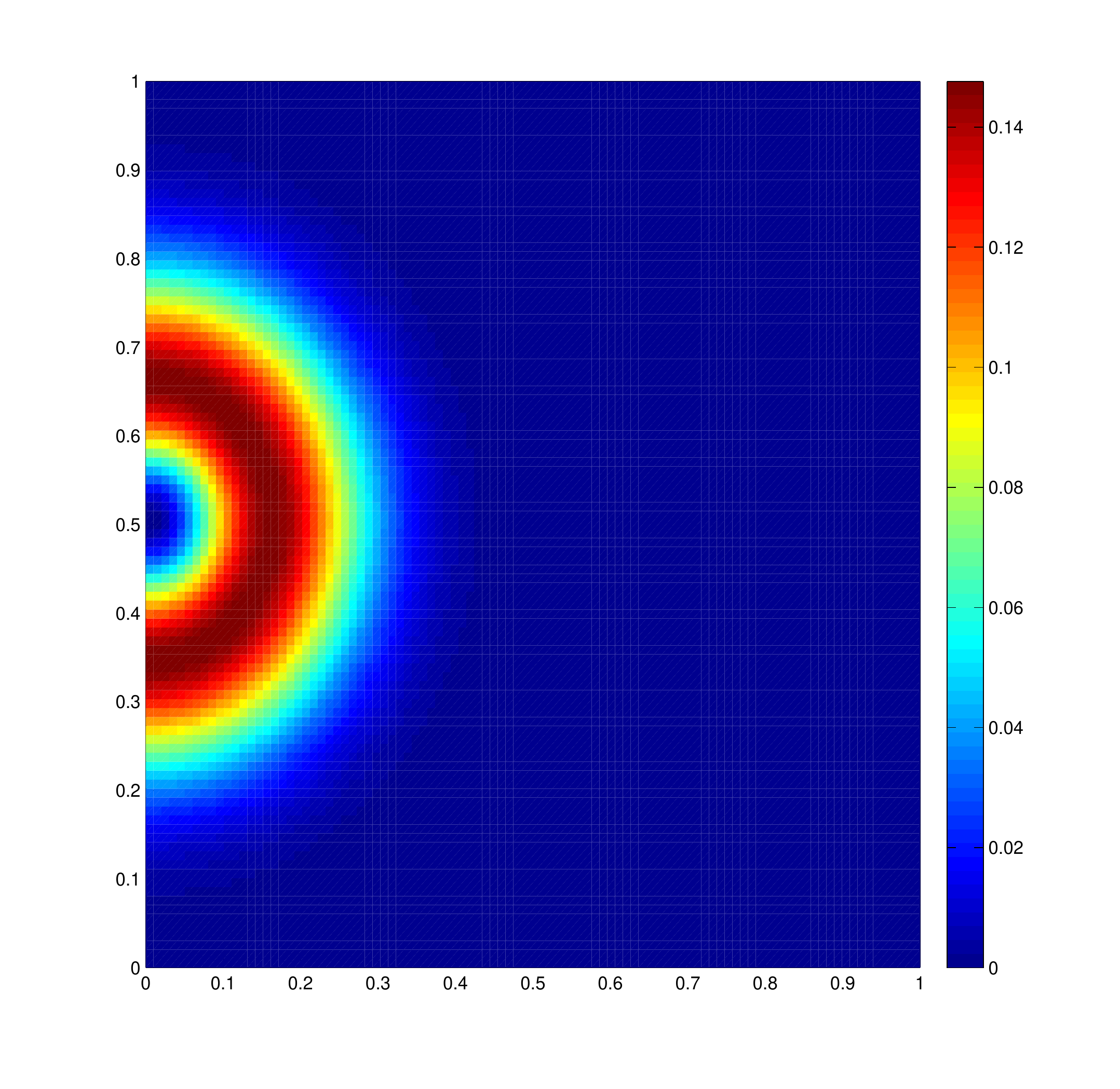}}\\
    \subfigure[SBP4, half rotation]{\includegraphics[width=0.45\linewidth]{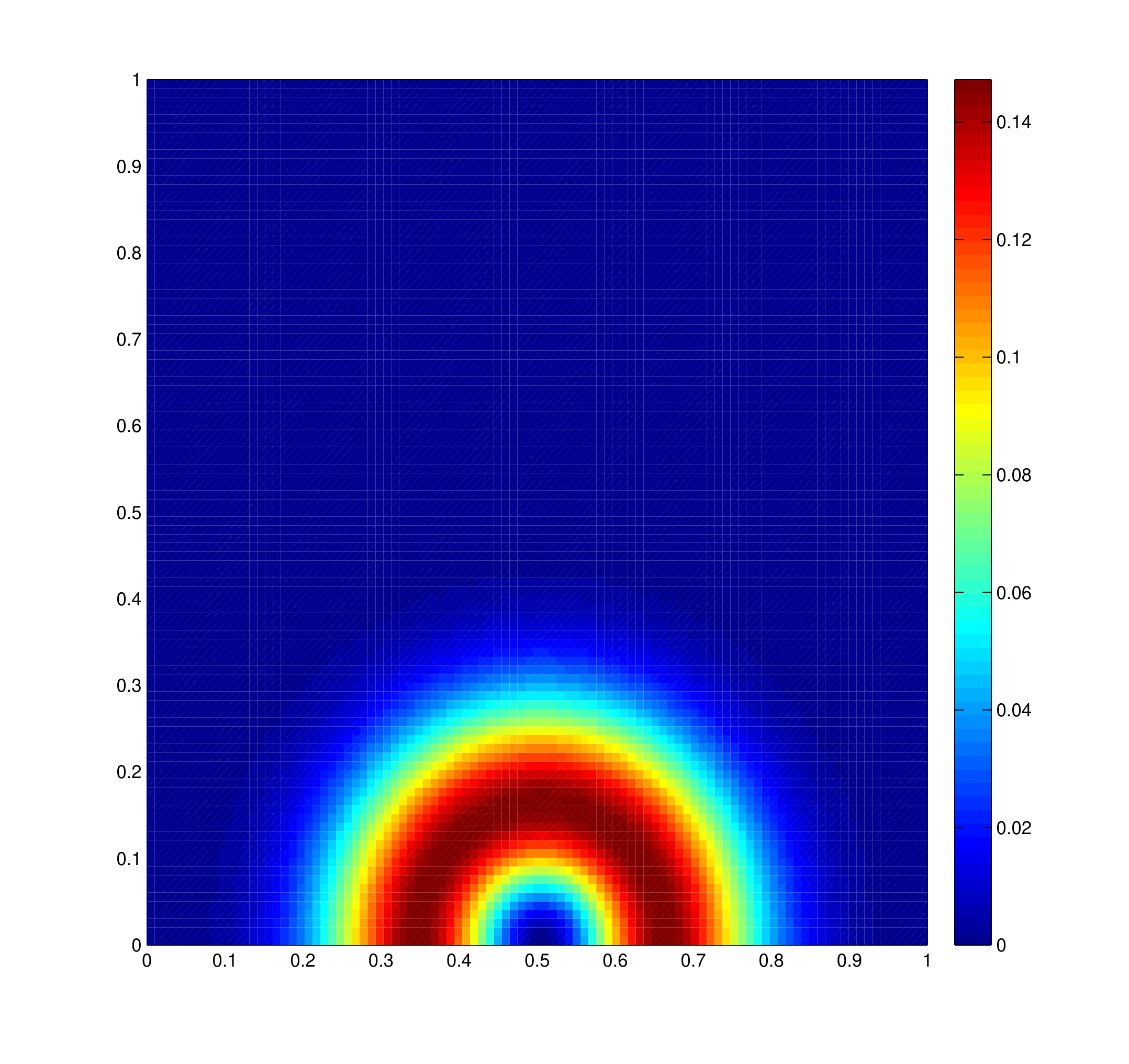}}
    \subfigure[SBP4, full rotation]{\includegraphics[width=0.45\linewidth]{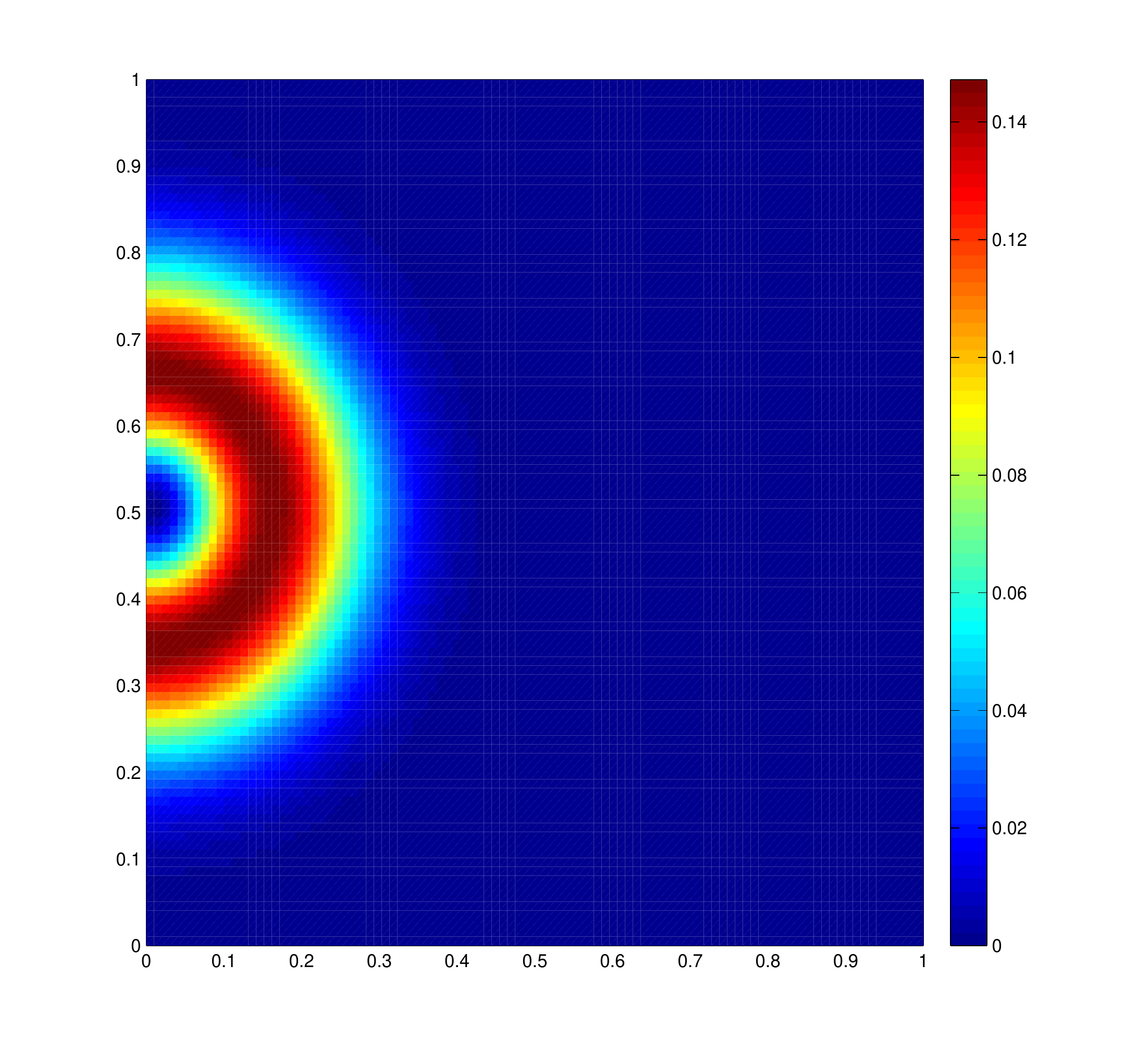}}
\caption{Numerical Experiment $2$: $|\B| = \sqrt{\B_1^2 + \B_2^2}$ for the SBP2 and SBP4 schemes on a $100 \times 100$ mesh.}
\protect \label{fig:2}
\end{figure}
The figure shows that both the SBP2 and SBP4 schemes resolve the solution quite well and maintain the shape of the hump. Furthermore, the boundary interactions are resolved in a stable and accurate manner, showing that the choice of boundary closures was proper. The errors in $L^2$ are shown in tables \ref{tab:3} and \ref{tab:div2} and we observe that the errors are quite low and the correct rates of convergence are obtained. The results were very similar to those obtained in numerical experiment $1$.  
\begin{table}[htbp]
\centering
\begin{tabular}{c|rrrr}
 Grid size & $SBP2$ & rate & $SBP4$  & rate\\
\hline
20$\times$20    & 5.5e{-2}     &          & 1.5e{-2}    &   \\
40$\times$40    & 1.0e{-2}     & 2.4      & 1.9e{-3}    & 2.9 \\
80$\times$80    & 2.3e{-3}     & 2.2      & 1.6e{-4}    & 3.5  \\
160$\times$160  & 5.4e{-4}     & 2.0      & 1.9e{-5}    & 3.1 \\ 
320$\times$320  & 1.3e{-4}     & 2.0         & 2.5e{-6}    & 3.0  
\end{tabular}
\caption{
Relative percentage errors in $L^2$ and rates of convergence for numerical experiment $2$.} 
\label{tab:3}
\end{table}

\begin{table}[htbp]
\centering
\begin{tabular}{c|rrrr}
Grid size& $SBP2$ & rate & $SBP4$ & rate\\
\hline
20$\times$20   & 5.6e{-2}         &                      & 9.8e{-2}     &  \\
40$\times$40   & 3.3e{-2}         & 0.8                  & 3.7e{-2}     & 1.4 \\
80$\times$80   & 9.3e{-3}         & 1.9                  & 2.3e{-3}     & 4.0 \\
160$\times$160 & 2.1e{-3}         & 2.1                  & 2.5e{-4}     & 3.2 \\
320$\times$320 & 7.4e{-4}         & 1.5                  & 4.1e{-5}     & 2.6
\end{tabular}
\caption{Divergence errors in $l^2$ and rates of convergence at time $t = 2\pi$ for numerical experiment $2$.} 
\label{tab:div2}
\end{table}

This shows that the SAT technique for imposing
boundary conditions weakly works very well even with complicated
boundary data.

\subsection*{Numerical Experiment $3$}

In the first two examples, we verified the accuracy of our schemes by manufactured solutions, using specific forcing functions. Next, we compute solutions of the magnetic induction equations in its original form \eqref{eq:main} (without any forcing) to illustrate the role of the resistivity, $\epsilon$, in driving the dynamics. We use the same velocity field and initial data as in the previous two numerical experiments. The domain is $[-1,1] \times [-1,1]$. We compute solutions till $t = 2 \pi$ with two different values of resistivity. The results with $\epsilon = 0.05$ and $\epsilon = 0.001$ are shown in figure \ref{fig:3}.
\begin{figure}[htbp]
  \centering
    \subfigure[SBP2, $\epsilon = 0.05$]{\includegraphics[width=0.45\linewidth]{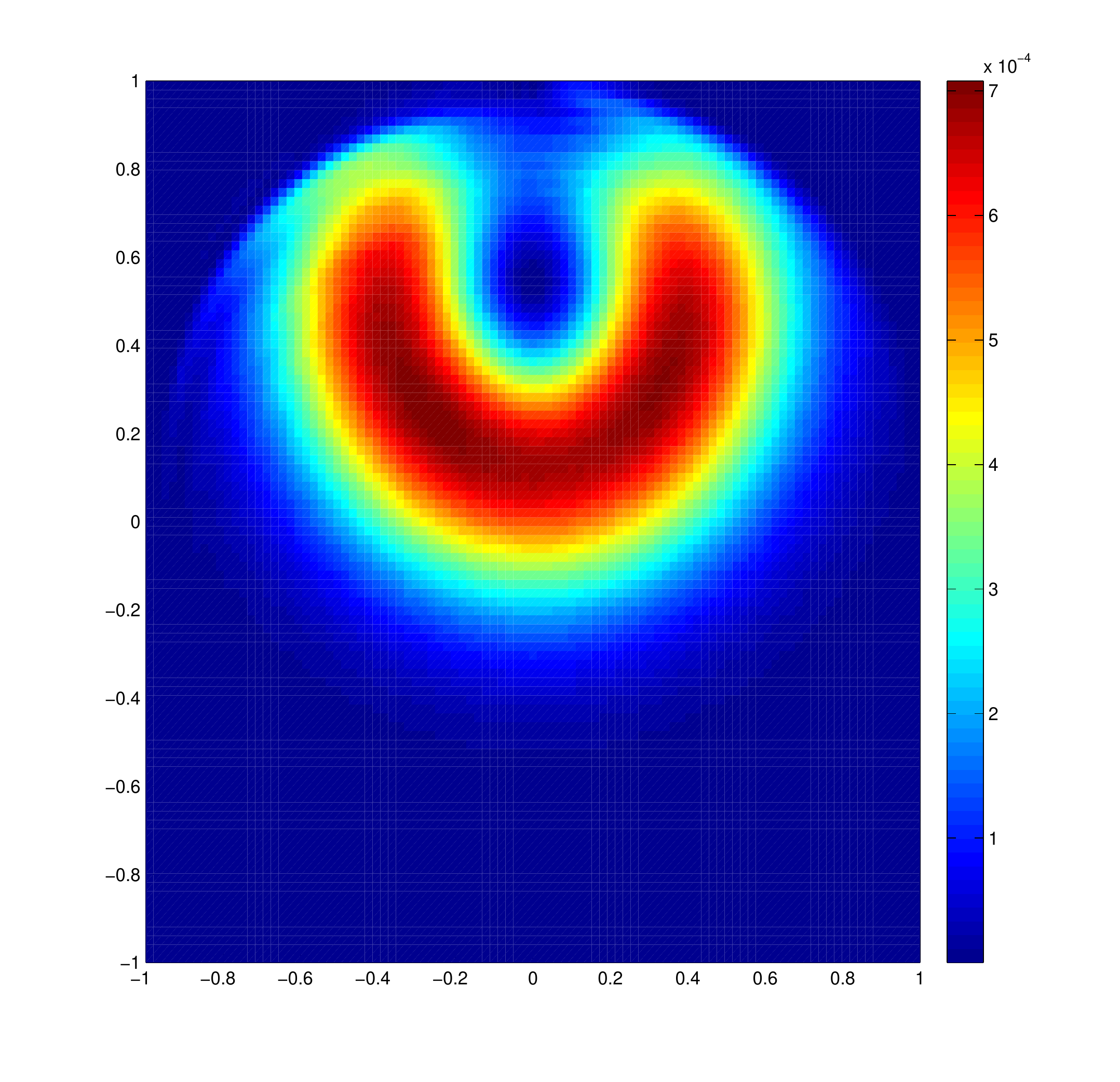}} 
    \subfigure[SBP2, $\epsilon = 0.001$]{\includegraphics[width=0.45\linewidth]{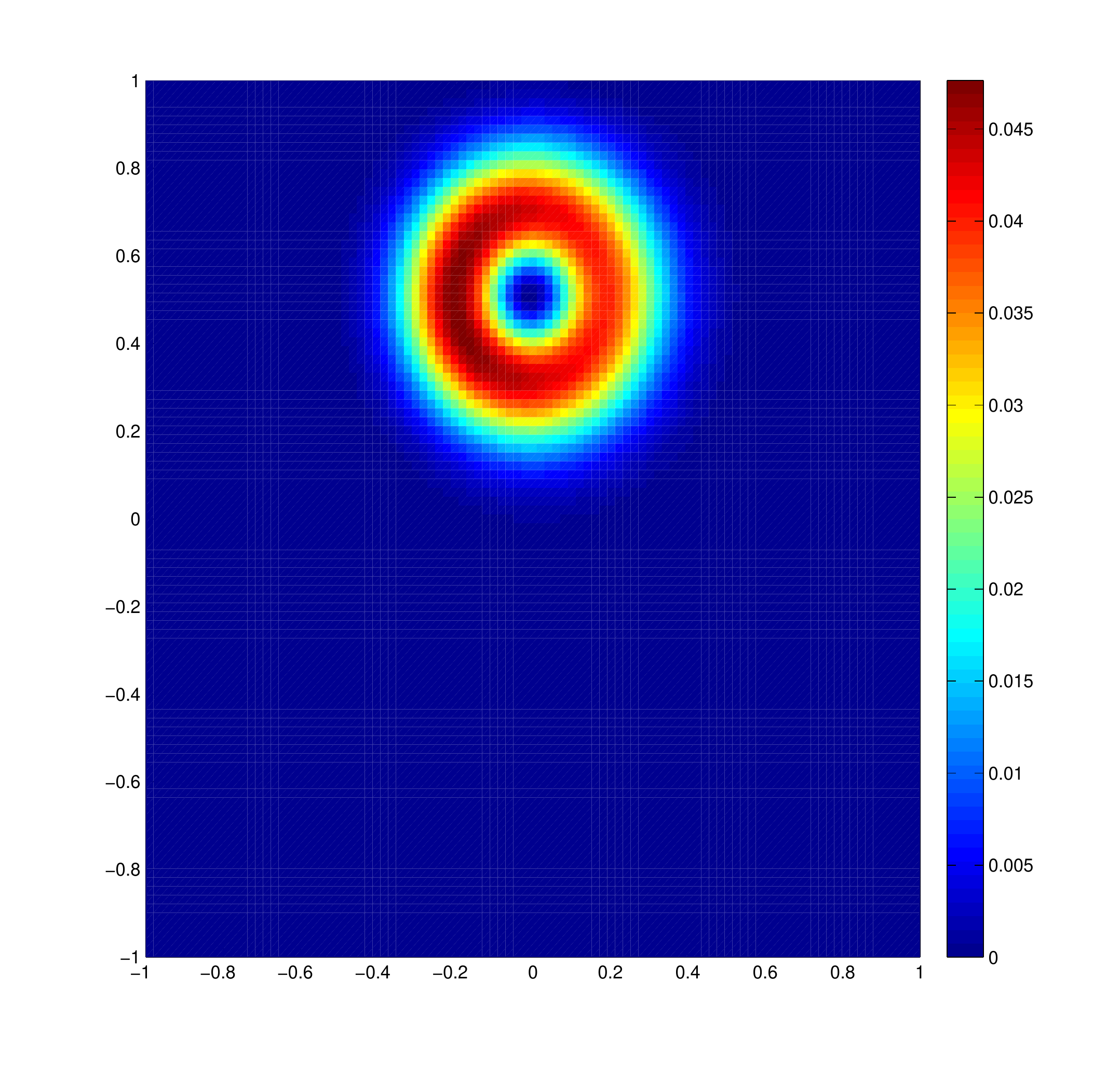}}\\
    \subfigure[SBP4, $\epsilon = 0.05$]{\includegraphics[width=0.45\linewidth]{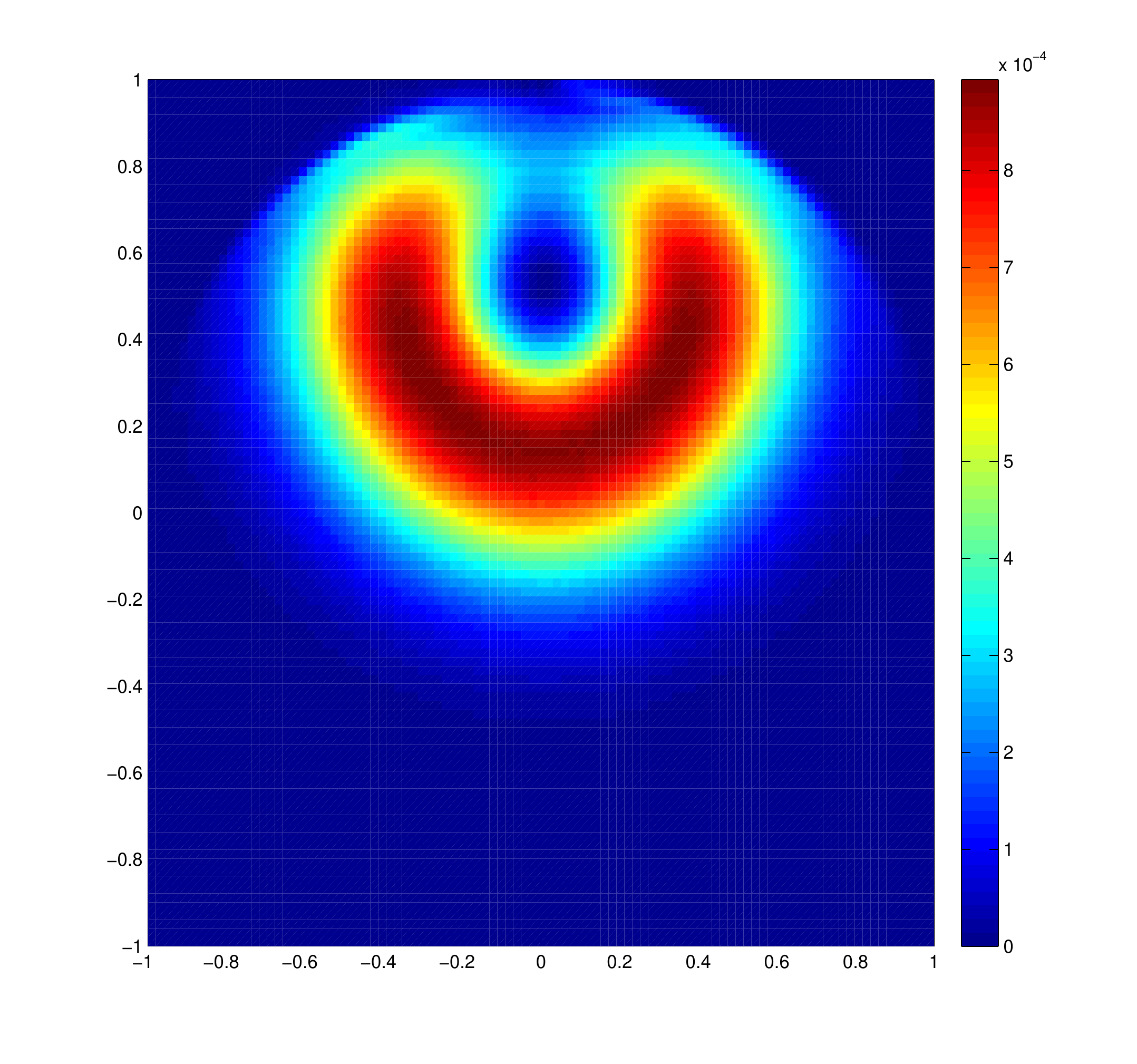}}
    \subfigure[SBP4, $\epsilon = 0.001$]{\includegraphics[width=0.45\linewidth]{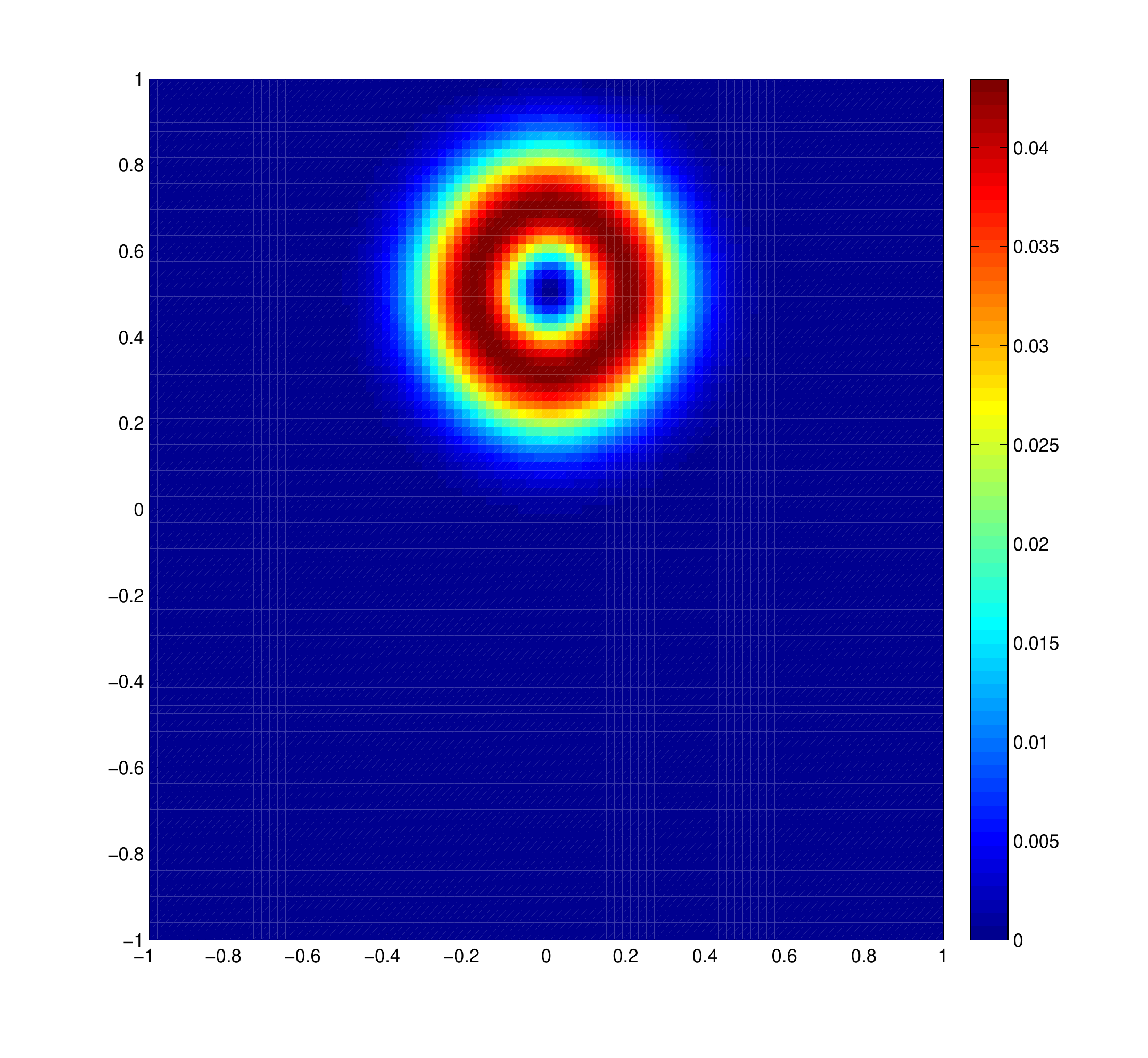}}
\caption{Numerical Experiment $3$: $|\B| = \sqrt{\B_1^2 + \B_2^2}$ for the SBP2 and SBP4 schemes on a $100 \times 100$ mesh.}
\protect \label{fig:3}
\end{figure} 
In the absence of exact solutions for $\B$, we can only compare qualitative features in this case. For low resistivities like $\epsilon = 0.001$, the problem is very close to its inviscid version and the hump doesn't show much distortion during the rotation. The results in this case are very similar to the ones for the inviscid magnetic induction equations presented in \cite{unsm}. The SBP4 scheme is slightly sharper (and hence more accurate) than the SBP2 scheme. Taking a higher value of resistivity $\epsilon = 0.05$, the viscous term starts playing an important role in the dynamics and the hump is expected to be smeared out. This is clearly shown in figure \ref{fig:3}. The results of SBP2 and SBP4 schemes are very similar in this case. 

Since we do not enforce the divergence constraint excactly, we can use divergence errors as a quantitative measure. We know that the initial ${\rm div}(\B)=0$ and it should remain so during the computation.  In tables \ref{tab:div31} and \ref{tab:div32}, we display the divergence errors and their convergence rates for two different values of $\epsilon$. The convergence rates are as expected and we note that the errors are much lower for the higher order scheme. Thus, this experiment illustrates that both schemes are quite robust and efficient for different values of the resistivity.
\begin{table}[htbp]
\centering
\begin{tabular}{c|rrrr}
Grid size& $SBP2$ & rate & $SBP4$ & rate\\
\hline
20$\times$20   & 1.0e{0}          &                      & 7.3e{-1}     &  \\
40$\times$40   & 8.0e{-1}         & 0.4                  & 1.2e{-1}     & 2.6 \\
80$\times$80   & 2.7e{-1}         & 1.6                  & 8.2e{-3}     & 3.8 \\
160$\times$160 & 7.0e{-2}         & 2.0                  & 1.0e{-3}     & 3.0 \\
320$\times$320 & 2.5e{-2}         & 1.5                  & 1.7e{-4}     & 2.6
\end{tabular}
\caption{Divergence errors in $L^2$ and rates of convergence at time $t = 2
 \pi$ for numerical experiment $3$ with $\epsilon=0.001$.} 
\label{tab:div31}
\end{table}
\begin{table}[htbp]
\centering
\begin{tabular}{c|rrrr}
Grid size& $SBP2$ & rate & $SBP4$ & rate\\
\hline
20$\times$20   & 7.3e{-1}         &                      & 1.4e{-2}     &  \\
40$\times$40   & 5.0e{-2}         & 0.8                  & 2.2e{-3}     & 2.7 \\
80$\times$80   & 1.1e{-2}         & 2.2                  & 1.7e{-4}     & 3.7 \\
160$\times$160 & 2.9e{-3}         & 1.9                  & 2.1e{-5}     & 3.1 \\
320$\times$320 & 9.7e{-4}         & 1.6                  & 3.4e{-6}     & 2.6
\end{tabular}
\caption{Divergence errors in $L^2$ and rates of convergence at time $t = 2
 \pi$ for numerical experiment $3$ with $\epsilon=0.05$.} 
\label{tab:div32}
\end{table}

\section{Conclusion}
\label{sec:conc}
We have presented finite difference schemes for the magnetic induction equations with resistivity. These equations arise as a sub model in the resistive MHD equations of plasma physics. We have shown that the 
symmetric form  \eqref{eq:viscousinduc1} of the magnetic induction equations with resistivity is well-posed with general initial data and both Dirichlet boundary conditions as well as mixed boundary conditions. SBP-SAT based finite difference schemes were designed for the initial-boundary-value problem corresponding to the magnetic induction equations with resistivity. These schemes were based on the 
form \eqref{eq:viscousinduc1} and use SBP finite difference operators to approximate spatial derivatives and an SAT technique for implementing boundary conditions. The resulting schemes are high-order accurate and shown to be energy stable. 

The schemes were tested on numerical experiments illustrating both the stability as well as high-order of accuracy. We also the use of the divergence errors as a measure of the accuracy of the solution. The results show that the SBP-SAT approach is a promising method to simulate initial boundary value problems for more complicated equations like the resistive MHD equations. 

\section*{Appendix}

 For the sake of completeness, here we will present all the SBP operators used in the analysis. We consider second and third order accurate finite difference approximations.

\subsection*{ First order accuracy at the boundary:}
The discrete norm $P$ and the discrete second order accurate SBP operator $P^{-1}Q$ approximating $\frac{d}{dx}$ are given by
$$
P = h
\begin{pmatrix}
  \frac{1}{2} &       &        &             & \\
              & 1     &        &             & \\
              &       &\ddots   &            & \\
              &       &         &1           &   \\ 
              &       &          &           & \frac{1}{2}
\end{pmatrix}, \quad 
P^{-1}Q = \frac{1}{h}
\begin{pmatrix}
  -1          & 1           &                   &                    & \\
  -\frac{1}{2}& 0           &\frac{1}{2}        &                    & \\
              & \ddots      &\ddots             & \ddots             & \\
              &             & -\frac{1}{2}      &0                   & \frac{1}{2}   \\ 
              &             &                   & -1                 & 1
\end{pmatrix}
$$
We have used the operator $P^{-1}QP^{-1}Q$ to approximate $\frac{d^2}{dx^2}$.
\subsection*{ Second order accuracy at the boundary:}

The discrete norm $P$ is defined as
$$
P = h
\begin{pmatrix}
  \frac{17}{48} &                   &                 &                  &        &        &    & \\
                & \frac{59}{48}     &                 &                  &        &        &    & \\
                &                   &\frac{43}{48}    &                  &        &        &    & \\
                &                   &                 &\frac{49}{48}     &        &        &    & \\ 
                &                   &                 &                  & 1      &        &    & \\
                &                   &                 &                  &        & \ddots &    &
\end{pmatrix}
$$
The discrete difference SBP operator $P^{-1}Q$ approximating $\frac{d}{dx}$ is given by
$$
P^{-1}Q = \frac{1}{h}
\begin{pmatrix}
 -\frac{24}{17} & \frac{59}{34}  &-\frac{4}{17}   & -\frac{3}{34}  &   0          &   0         &  0          &      &\\
  -\frac{1}{2}  & 0              & \frac{1}{2}    &  0             &   0          &   0         &  0          &      &\\
  \frac{4}{43}  & -\frac{59}{86} & 0              & \frac{59}{86}  &-\frac{4}{43} &   0         &  0          &      &\\
  \frac{3}{98}  & 0              & -\frac{59}{98} &   0            &\frac{32}{49} &-\frac{4}{49}&  0          &      &\\ 
  0             & 0              & \frac{1}{12}   & -\frac{2}{3}   & 0            &\frac{2}{3}  &-\frac{1}{12}&      &\\
                &                &                & \ddots         & \ddots       & \ddots      & \ddots      &\ddots&
\end{pmatrix}
$$


\begin{thebibliography}{10}
  
\bibitem{BIS1} D. Biskamp.
\newblock Nonlinear magnetohydrodynamics.
\newblock {\it Cambridge monographs on plasma physics}, Cambridge university press, 1993.


\bibitem{unsm} U.~Koley, S.~Mishra, N.H.~Risebro and M.~Sv\"{a}rd.  \newblock Higher order finite
difference schemes for the magnetic induction
equations.  \newblock {\em Preprint,} BIT Numerical analysis, to appear.
  
\bibitem{tg} T.C.~Warburton and G.E.~Karniadakis.  \newblock A discontinuous Galerkin method for the
viscous MHD equations.  \newblock {\em Journal of 
computational physics.} 152,608-641, 1999.

  
\bibitem{BlBr1} J.U.~Brackbill and D.C.~Barnes.  \newblock The
  effect of nonzero $\Div B$ on the numerical solution of the
  magnetohydrodynamic equations.  \newblock {\em J. Comp. Phys.},
  35:426-430, 1980.

  
\bibitem{fkrsid1} F.~Fuchs, K.H.~Karlsen, S.~Mishra and N.H.~Risebro.
  \newblock Stable upwind schemes for the Magnetic Induction equation.
  \newblock {\em Preprint,}, M2AN. Math. model. Num. Anal, to appear.

\bibitem{frsid1} F. Fuchs, S. Mishra and N. H. Risebro.  \newblock
  Splitting based finite volume schemes for the ideal MHD equations.
  \newblock {\em J. Comp. Phys.,} 228 (3), 2009, 641-660.
  
\bibitem{Pano} K.G.~Powell, P.L.~Roe. T.J.~Linde, T.I.~Gombosi
  and D.L.~De~Zeeuw, \newblock A solution adaptive upwind scheme for
  ideal MHD.  \newblock J. Comp. Phys, 154(2), 284 - 309, 1999
  

\bibitem{svardsid3} S.~ Mishra and M. Sv\"ard.  \newblock On stability
  of numerical schemes via frozen coefficients and magnetic induction
  equations.  \newblock {\em Preprint}, Submitted.

\bibitem{Svard} M.~Sv\"{a}rd. \newblock On
  coordinate transformations for summation-by-parts
  operators. \newblock {\em J.~Sci.~Comput.} 20(2004), 29-42.
  

\bibitem{SvardNordstrom} M.~Sv{\"a}rd and J.~Nordstr{\"o}m.
\newblock On the order of accuracy for difference approximations of 
initial-boundary value problems 
\newblock {\em Journal of Computational Physics}, 218 (2006) 333–352.



  
\bibitem{TF1} M.~Torrilhon and M.~Fey.  \newblock
  Constraint-preserving upwind methods for multidimensional advection
  equations.  \newblock {\em SIAM. J. Num. Anal.}, 42(4):1694-1728,
  2004.


\bibitem{Toth1} G.~ Toth.  \newblock The $\Div B =0$ constraint in
  shock capturing magnetohydrodynamics codes.  \newblock {\em
    J. Comp. Phys.},161:605-652, 2000.


\end{thebibliography}
\end{document}